%% file: aug_mixed_20231227.tex
\RequirePackage{ifpdf} % running on pdfTeX?
\ifpdf
\documentclass[10pt,pdftex]{amsart}%{article}
\else
\documentclass[10pt,dvips]{amsart}%{article}
\fi

\ifpdf
\usepackage{lmodern}
\fi
\usepackage[bookmarks, % add hyperlinks
colorlinks=false, backref=false,plainpages=false]{hyperref}
\usepackage[T1]{fontenc}
\usepackage[latin1]{inputenc}
\usepackage[english]{babel}
\usepackage{graphicx}
\usepackage{url}
\usepackage{graphics}
\usepackage{epsfig}
\usepackage{amsmath,amssymb,amsthm}

\renewcommand{\theequation}{\thesection.\arabic{equation}}

\newtheorem{thm}{Theorem}[section]

\newtheorem{lem}[thm]{Lemma}

\newtheorem{assumption}[thm]{Assumption}
\newtheorem{rem}[thm]{Remark}

\begin{document}
\input defs.tex
\def\grad{{\nabla}}

\def\calS{{\cal S}}
\def\calT{{\cal T}}
\def\cA{{\mathcal A}}
\def\cB{{\cal B}}
\def\cD{{\mathcal{D}}}

\def\cH{{\cal H}}
\def\ba{{\mathbf{a}}}

\def\beps{{\mathbf{\epsilon}}}

\def\cM{{\mathcal{M}}}
\def\cN{{\mathcal{N}}}
\def\cT{{\mathcal{T}}}
\def\cE{{\mathcal{E}}}
\def\cP{{\mathcal{P}}}
\def\cF{{\mathcal{F}}}

\def\cB{{\mathcal{B}}}
\def\cG{{\mathcal{G}}}

\def\cL{{\mathcal{L}}}
\def\cJ{{\mathcal{J}}}
\def\cV{{\mathcal{V}}}
\def\cW{{\mathcal{W}}}

\newcommand{\lJump}{[\![}
\newcommand{\rJump}{]\!]}
\newcommand{\jump}[1]{[\![ #1]\!]}

\newcommand{\sd}{\bsigma^{\Delta}}
\newcommand{\rd}{\brho^{\Delta}}

\newcommand{\eps}{\epsilon}
\newcommand{\R}{\rm I\kern-.19emR}

\title [Least-Squares versus Partial Least-Squares Finite Element Methods: Robust Error Analysis]{Least-Squares versus Partial Least-Squares Finite Element Methods: Robust A Priori and A Posteriori Error Estimates of Augmented Mixed Finite Element Methods}
\author[Y. Liang and S. Zhang]{Yuxiang Liang and Shun Zhang}
\address{Department of Mathematics, City University of Hong Kong, Kowloon Tong, Hong Kong, China}
\email{yuxiliang7-c@my.cityu.edu.hk, shun.zhang@cityu.edu.hk}
\thanks{This work was supported in part by
Research Grants Council of the Hong Kong SAR, China, under the GRF Grant Project No. CityU 11316222}
\date{\today}


\maketitle
\begin{abstract}
In this paper, for the generalized Darcy problem (an elliptic equation with discontinuous coefficients), we study a special  partial Least-Squares (Galerkin-least-squares) method, known as the augmented mixed finite element method, and its relationship to the standard least-squares finite element method (LSFEM). Two versions of augmented mixed finite element methods are proposed in the paper with robust a priori and a posteriori error estimates. Augmented mixed finite element methods and the standard LSFEM uses the same a posteriori error estimator: the evaluations of numerical solutions at the corresponding least-squares functionals. As partial least-squares methods, the augmented mixed finite element methods are more flexible than the original LSFEMs.  As comparisons, we discuss the mild non-robustness of a priori and a posteriori error estimates of the original LSFEMs. A special case that the $L^2$-based LSFEM is robust is also presented for the first time. Extensive numerical experiments are presented to verify our findings.
\end{abstract}

%\maketitle
\section{Introduction}\label{intro}
 We consider the following elliptic equation with possible discontinuous coefficients (a generalized Darcy's problem):
\begin{equation}\label{eq_Darcy}
\left\{
\begin{array}{lllll}
\gradt \bsigma    & =& g, & \mbox{in } \O~~\mbox{the constitutive equation},
 \\[1mm]
A\nabla u+ \bsigma  & =& A\bff, & \mbox{in } \O~~\mbox{the equilibrium equation},
\\[1mm]
u &=& 0 & \mbox{on } \Gamma_D\\
\bsigma\cdot\bn & = & 0 &\mbox{on } \Gamma_N,
\end{array}
\right.
\end{equation}
The domain $\O$ is a bounded, open, connected subset of $\mathbb{R}^d (d = 2 \mbox{ or } 3)$ with a Lipschitz continuous boundary $\p\O$. We partition $\p\O$ into two open subsets $\G_D$ and $\G_N$, such that $\p\O = \overline{\G_D} \cup \overline{\G_N}$ and $\G_D\cap \G_N =\emptyset$. For simplicity, we assume that $\G_D$ is not empty (i.e., $\mbox{meas}(\G_D) \neq 0$ ) and is connected.
We assume that the diffusion coefficient matrix $A \in L^{\infty}(\O)^{d\times d}$ is a given $d\times d$ tensor-valued function;  the matrix $A$ is uniformly symmetric positive definite: there exist positive constants $0 < \Lambda_0 \leq \Lambda_1$ such that
\beq\label{A}
\Lambda_0 \by^T\by \leq \by^T A \by \leq \Lambda_1 \by^T\by
\eeq
for all $\by\in \mathbb{R}^d$ and almost all $x\in \O$. 
The righthand sides $g\in L^2(\O)$ and $\bff\in L^2(\O)^d$. %The first two equations of \eqref{eq_Darcy} can be views in the $L^2$ sense. 

There are several variational formulations for \eqref{eq_Darcy}. The standard conforming finite element method tries to approximate $u$ in the finite element subspace of its natural space $H^1(\O)$, see \cite{Ciarlet:78,Braess:07,BrSc:08}. If a good approximation of $\bsigma$ in its natural space $H(\divvr)$ is sought, then one can use the mixed finite element approximation based on a dual mixed formulation, see for example, \cite{BBF:13}. In order to approximate both $u$ and $\bsigma$ in their intrinsic spaces $H^1(\O)$ and $H(\divvr)$, respectively, one natural choice is the least-squares finite element method (LSFEM).

The least-squares variational principle and the corresponding LSFEM based on a first-order system reformulation have been widely used in numerical solutions of partial differential equations; see, for example \cite{CLMM:94,CMM:97,Jiang:98,BG:09,CS:04,CLW:04,CFZ:15,LZ:18,LZ:19,QZ:20}. Compared to the standard variational formulation and the related finite element methods, the first-order system LSFEMs have several known advantages, such as the discrete problem is stable without the inf-sup condition of the discrete spaces and mesh size restriction; the first-order system is physically meaningful, thus important physical qualities such as displacement, flux, and stress can be approximated in their intrinsic spaces; and the least-squares functional itself is a good built-in a posteriori error estimator. In Section 4 of \cite{Zhang:23}, as an example of an elliptic equation with an $H^{-1}$-righthand side, the first-order system LSFEM is studied for \eqref{eq_Darcy}. 

On the other hand,  when applying LSFEMs to \eqref{eq_Darcy}, we also face an important issue: robustness. For a coefficient-dependent problem ($A$ for \eqref{eq_Darcy}), the robustness of both the a priori and a posteriori estimates, i.e., genetic constants that appeared in the estimates are independent of the coefficients, is of crucial importance. For the a priori error estimate, it is well known that the model problem \eqref{eq_Darcy} may only have $H^{1+s}$ regularity, with possibly very small $s>0$, see for example, Kellogg \cite{Kellogg:74}. In \cite{CHZ:17}, it is noted that the a priori analysis using the global regularity constant $s$ is not satisfactory since most of the time, the solution only has a low regularity at the elements attached to the singularity but can be very smooth away from the singularity. Thus, we should study the a priori analysis under the local regularity assumption. The a priori error estimate using local regularity is also the base for adaptive finite element methods to achieve an equal-discretization error distribution. For the a posteriori error estimate, we want the a posteriori error estimator with the efficiency constant and the reliability constant to be independent of the parameter of the problem. Unfortunately, both a priori and a posteriori error estimates of the LSFEM applying to \eqref{eq_Darcy} with discontinuous coefficients are not robust; see our discussion in Sections \ref{comparison_1} and \ref{comparison_2}. Beside the above mentioned robustness issue, we often need extra regularity for the right-hand side $g$ in the standard $L^2$-based LSFEM since all the errors of the LSFEM are measured in a combined norm. On the other hand, the standard mixed finite element method for \eqref{eq_Darcy} does not require this, see discussions in \cite{Zhang:20mixed}.

Besides the full bonafide LSFEM, another idea to apply the least-squares philosophy is the so-called Galerkin-Least-Squares (GaLS) method. The GaLS method is a method combining the least-squares and Galerkin methods. Some least-squares terms are added to the original variational formulation to enhance the stability in GaLS methods.  We consider a special GaLS method, the augmented mixed finite element method, in this paper. The central idea of the augmented mixed method is adding consistent least-squares terms to the original mixed formulation to guarantee coercivity or stability.  The first augmented mixed finite element method is introduced in \cite{MH:02}. Earlier contributions of GaLS methods can be found in \cite{Franca:87,FH:88}. The group of Gatica made many important and seminal contributions on developing the augmented mixed finite element method to many problems, see for example \cite{Gatica:06,BGGH:06,FGM:08,CGOT:16,CGO:17,AGR:20}.

As a partial least-squares method, the augmented mixed finite element method shares many properties with the classic LSFEM. It is also stable and approximates the physical quantities in their native spaces. Moreover, as we will see later in the paper, the a posteriori error estimator of the augmented mixed finite element method is a least-squares error estimator. Furthermore, since we only partially use the least-squares principle, the system is more flexible and we can show the robustness of both a priori and a posteriori error estimates. 

This paper proposes two versions of robust augmented mixed finite element methods with two choices of least-squares terms based on the constitutive equation. The first is a simple $L^2$ least-squares term and the second is a mesh-weighted least-squares term. We show the robust and local optimal a priori error estimates for both methods. For the mesh-weighted augmented mixed finite element method, we also show that optimal error can be achieved without requiring high regularity of the righthand side. For both methods, we derive robust a posteriori error estimates. In fact, the error estimators of both versions of augmented mixed finite element methods are least-squares error estimators with corresponding least-squares functionals. As comparisons, we discuss the robustness of two corresponding LSFEMs. For the $L^2$-based LSFEM, we show that the a priori and a posteriori error estimates are robust only under a very special condition. In general, they are not robust. For the mesh-weighted LSFEM, the results are much worse than those of augmented mixed finite element methods since its coercivity constant depends on the mesh size. 

The robust (but not local optimal) a priori  and robust residual type a posteriori error estimate for the energy norm was obtained for the conforming FEM \cite{BeVe:00}. Robust and local optimal {\it a priori} error estimates have also been derived for mixed FEM in \cite{Zhang:20mixed} and nonconforming FEM and discontinuous Galerkin FEM in \cite{CHZ:17} without a restrictive assumption on the distribution of the coefficients. We also present a detailed discussion of the robust and local optimal {\it a priori} error estimate for the conforming FEM in \cite{CHZ:17}. For recovery-based error estimators, robust a posteriori error estimates are obtained by us in \cite{CZ:09,CZ:10a,CYZ:11}. Robust equilibrated error estimators were developed by us in \cite{CZ:12,CCZ:20,CCZ:20mixed}. Robust residual-type of error estimates without a restrictive assumption on the distribution of the coefficients is developed for nonconforming and DG approximations in \cite{CHZ:17}.

In the original paper \cite{MH:02}, both $u$ and $\bsigma$ are approximated by continuous finite elements. As mentioned in \cite{CZ:09,CHZ:17}, the flux $\bsigma$ is not continuous for discontinuous coefficients. Thus the continuous finite element is not a good candidate for the approximation. In \cite{BHMM:05}, a mixed discontinuous Galerkin method is developed using the stabilization in \cite{MH:02}. In \cite{CL:08}, several different stabilization formulations are proposed, with one formulation being our first augmented mixed method, although the authors still emphasized using standard conforming finite elements. In \cite{BCG:15}, $H(\divvr)$ and $H^1$-conforming finite elements are used. In \cite{BCG:15}, an a posteriori error estimator is proposed for the first augmented mixed formulation. The error estimator in \cite{BCG:15} is actually a least-squares error estimator. The robustness of a priori and a posteriori error estimates are not discussed in all previous papers.

The paper is organized as follows. Section 2 describes notations, the function spaces, and local interpolation results. The generalized Darcy problem and the augmented formulations are presented in Section 3, including the robust Cea's lemma. In Section 4, we discuss simplified assumptions on the coefficient $A$ and the quasi-monotonicity assumption and robust quasi-interpolation based on this assumption. In Sections 5 and 6, we present robust a priori and a posteriori error analyses for the first and second augmented mixed formulations, respectively. Connections and comparisons with the corresponding LSFEMs are also discussed in Sections 5 and 6. Numerical experiments are presented in Section 7 to verify the findings in the paper. We make some concluding remarks in Section 8. 

\section{Preliminaries}%\label{intro}
\setcounter{equation}{0}
%\subsection{Notations and the function spaces}
%\setcounter{equation}{0}
%Let $\O$ be a bounded, open, connected subset of $\mathbb{R}^d (d = 2 \mbox{ or } 3)$ with a Lipschitz continuous boundary $\p\O$. We partition the boundary of the domain $\O$ into two open subsets $\G_D$ and $\G_N$, such that $\p\O = \overline{\G_D} \cup \overline{\G_N}$ and $\G_D\cap \G_N =\emptyset$. For simplicity, we assume that $\G_D$ is not empty (i.e., $\mbox{meas}(\G_D) \neq 0$ ) and is connected.

We use the standard notations and definitions for the Sobolev spaces $H^s(\O)$ for $s\ge 0$. The standard associated inner product is denoted by $(\cdot , \, \cdot)_{s,\O}$, and its norm is denoted by $\|\cdot \|_{s,\O}$. The notation $|\cdot|_{s,\O}$ is used for the semi-norm.  (We suppress the superscript $d$ because the dependence on dimension will be clear by context. We also omit the subscript $\O$ from the inner product and norm designation when there is no risk of confusion.) For $s=0$, $H^s(\O)$ coincides with $L^2(\O)$.  The symbols $\gradt$ and $\nabla$ stand for the divergence and gradient operators, respectively. Set $H^1_D(\Omega):=\{v\in H^1(\Omega)\, :\, v=0\,\,\mbox{on }\Gamma_D\}$. We use the standard $H(\divvr;\O)$ space equipped with the norm
$
\|\btau\|_{H(\divvr;\,\O)}=\left(\|\btau\|^2_{0,\O}+\|\gradt\btau\|^2_{0,\O}
 \right)^\frac12.
$
Denote its subspace by
$
H_N(\divvr;\O)=\{\btau\in H(\divvr;\O)\, :\,
\btau\cdot\bn|_{\Gamma_N}=0\},
$
where $\bn$ is the unit vectors outward normal to the boundary $\p\O$. For simplicity, we use the following notation:
\beq
\bX := H_N(\divvr;\,\O) \times H^1_D(\Omega).
\eeq

Let $\cT = \{K\}$ be a triangulation of $\O$ using simplicial elements. The mesh $\cT$ is assumed to be regular.  Let $P_k(K)$ for $K\in\cT$ be the space of polynomials of degree $k$ on an element $K$.
Denote the standard linear and quadratic conforming finite element spaces by 
$$
S_{1,D} = \{v\in  H_D^1(\O): v|_K \in P_1(K),~\forall~K\in\cT \}
\quad\mbox{and}\quad
S_{2,D} = \{v\in H_D^1(\O): v|_T \in P_2(K),~\forall~K\in\cT\},
$$
respectively. 

Denote the local lowest-order Raviart-Thomas (RT) \cite{RT:77} on element $K\in\cT$ by $RT_0(K)=P_0(K)^d +\bx\,P_0(K)$. Then the lowest-order $\Hdiv$ conforming RT space with zero trace on $\Gamma_N$ is defined by
 \[%\begin{eqnarray*}%\label{S_h}
 RT_{0,N}=\{\btau\in H_N(\divvr;\Omega):
 \btau|_K\in RT_0(K)\,\,\,\,\forall\,\,K\in\cT\}.
 \] 
Similarly,  the lowest-order  $\Hdiv$-conforming Brezzi-Douglas-Marini (BDM) space with zero trace on $\Gamma_N$ is defined by
$$
BDM_{1,N} = \{\btau\in H_N(\divvr;\O): \btau|_K \in P_1(K)^d,~\forall~K \in\cT\}.
$$

We discuss some \noindent{\bf local approximation results} of the standard $S_{1,D}$ and $RT_{0,N}$ spaces.
By Sobolev's embedding theorem, $H^{1+s}(\O)$, with $s>0$ for two dimensions and $s > 1/2$ for three dimensions, is embedded in $C^0(\O)$. Thus, we can define the nodal interpolation $I^{nodal}_h$ of a function $v\in H^{1+s}(\O)$ with $I^{nodal}_hv \in S_{1,D}$ and $I^{nodal}_hv(z) = v(z)$ for a vertex $z\in \cN$. It is important to notice that the nodal interpolation is completely element-wisely defined.  
%For an element $K\in\cT$, by Sobolev’s embedding theorem, $H^{1+s}(K)$, with $s>0$ for two dimensions and $s > 1/2$ for three dimensions, is embedded in $C^0(K)$ and, hence, the nodal interpolation of a function $v\in H^{1+s}(K)$ is well-defined. For such $v\in C^0$, define $I^{nodal}v(\bx) = v(\bx)$, where $\bx$ is a node of the finite element mesh. 
We have the following local interpolation estimate for the linear nodal interpolation $I^{nodal}_h$ with local regularity $0<s_K\leq 1$ in two dimensions and $1/2<s_K\leq 1$ in there dimensions \cite{DuSc:80,CHZ:17}:
\beq \label{nodal_inter}
\|\nabla(v- I^{nodal}_h v)\|_{0,K} \leq C h_K^{s_K}|\nabla v|_{s_K,K} .
%\quad\mbox{and}\quad
%\|v- I^{nodal}_h v\|_{0,K} \leq C h_K^{1+s_K}|\nabla v|_{s_K,K}.
\eeq

Assume that $\btau\in L^r(\O)^d\cap H(\divvr;\O)$, and locally $\btau \in H^{s_K}(K)$ with the local regularity  $1/2<s_K\leq 1$. Let $I^{rt}_h$ be the canonical RT interpolation from $L^r(\O)^d\cap H_N(\divvr;\O)$ to $RT_{0,N}$. Then the following local interpolation estimates hold for local regularity $1/2<s_K\leq 1$ with the constant $C_{rt}$ being unbounded as $s_K\downarrow 1/2$ (see Chapter 16 of \cite{FE1}): 
\beq \label{RT_inter1}
\|\btau- I^{rt}_h\btau\|_{0,K} \leq C_{rt} h_K^{s_K}|\btau|_{s_K,K} \quad \forall K\in \cT.
\eeq
Due to the commutative property of the standard RT interpolation, if we further assuming that $\gradt \btau|_K \in H^{t_K}(K)$, $0<t_K\leq 1$, then
\beq \label{RT_inter2}
\|\gradt (\btau- I^{rt}_h\btau)\|_{0,K} \leq C h_K^{t_K}|\gradt\btau|_{t_K,K} \quad \forall K\in \cT.
\eeq
For $v\in H^3(\O)$, the following local interpolation result in standard for nodal interpolation $I_h$ in $S_{2,D}$,
\beq \label{nodal_inter2}
\|\nabla(v- I_h v)\|_{0,K} \leq C h_K^{2}|v|_{3,K}.
\eeq
Also, we have the following standard interpolation result for the $BDM_{1,N}$, assuming that $I^{bdm}_h$ is the standard $BDM_{1}$ interpolation,
\beq \label{BDM_inter1}
\|\btau- I^{bdm}_h\btau\|_{0,K} \leq C_{bdm} h_K^{2}|\btau|_{2,K} \quad \forall K\in \cT.
\eeq
If we further assuming that $\gradt \btau|_K \in H^{t_K}(K)$, $0<t_K\leq 1$, then
\beq \label{RT_inter2}
\|\gradt (\btau- I^{bdm}_h\btau)\|_{0,K} \leq C h_K^{t_K}|\gradt\btau|_{t_K,K} \quad \forall K\in \cT.
\eeq

The following mesh-dependent notation is used in the paper: for $r\geq 0$,
\beq
\|h^r v\|_0 = \left( \sum_{K\in\cT} h_K^{2r}\|v\|_{0,K}^2 \right)^{1/2} \quad\mbox{and}\quad 
(h^r v,w) =  \sum_{K\in\cT} h_K^r (v,w)_K, \quad
\forall v,w\in L^2(\O).
\eeq

\section{The generalized Darcy problem and the augmented mixed formulations}
\setcounter{equation}{0}

For the solution $u \in  H_D^1(\O)$, the flux $\bsigma = A \bff - A\nabla u$ is  in $L^2(\O)^d$, and $\gradt\bsigma = g$ is  in $L^2(\O)$. Thus,  $\bsigma \in H_N(\divvr\;\O)$. Multiplying an arbitrary $v\in  H_D^1(\O)$ on both sides of the first equation of \eqref{eq_Darcy}, taking integration by parts, and replacing $\bsigma$ by $A \bff - A\nabla u $, we have the standard variational problem of \eqref{eq_Darcy}: find $u\in H^1_D(\O)$, such that
\beq \label{primal}
(A\nabla u,\nabla v)=(g, v)+(A\bff, \nabla v)
\quad\forall v\in H_D^1(\O).
\eeq
%
%
%For simplicity, with respect to $\cT$, assume that $A$ is locally and mildly anisotropic in the sense that there exists a moderate size constant $c_K>0$ such that
%\beq
%\lambda
%\eeq
We have two mixed formulations.

\noindent{\bf Dual mixed method}:
Find $(\bsigma, u) \in H_N(\divvr,\O)\times L^2(\O)$ such that
\begin{equation}\label{eq_dm}
\begin{split}
(A^{-1} \bsigma, \btau )- (u, \gradt\btau) & =(\bff, \tau), \quad
\forall\btau \in H_N(\divvr;\O),\\
(\gradt\bsigma,v) & =  (g, v), \quad \forall v \in L^2(\O).
\end{split}
\end{equation}
{\bf Primal mixed method}: Find $(\bsigma, u) \in L^2(\O)^d \times H_D^1(\O)$ such that
\begin{equation}\label{eq_pm}
\begin{split}
(A^{-1}\bsigma, \btau)+(\nabla u, \btau) &=(\bff,\btau), \quad \forall\btau \in L^2(\O)^d \\
-(\bsigma, \nabla v) &=(g, v), \quad\forall v \in H_D^1(\O).
\end{split}
\end{equation}
The finite element approximations of the mixed methods require inf-sup stable finite element pairs.

In the augmented mixed method, we want to approximate $u$ in its natural space $H_D^1(\O)$ and keep the method stable without restricting the choice of finite element subspaces. Testing $\btau \in H_N(\divvr; \O)$ for the first equation of \eqref{eq_Darcy} and $v \in H_D^1(\O)$ for the second equation of \eqref{eq_Darcy}, we have
\beq\label{eq_mixed2}
(A^{-1}\bsigma, \btau)+(\nabla u, \btau)+(\gradt\bsigma, v)= (\bff, \btau)+(g, v), ~\forall (\btau, v) \in \bX. %H_N(\divvr ; \O) \times H_D^1(\O).
\eeq
We add two consistent least-squares terms to \eqref{eq_mixed2}:
\beq\label{eq_aurterm1}
-\k(A^{-1}\bsigma + \nabla u, \btau - A\nabla v) = -\k(\bff,\btau-A\nabla v)
~\mbox{ from the constitutive equation},
\eeq
and
\beq\label{eq_augterm2}
\mu(\a^{-1}\gradt\bsigma,\gradt\btau) = \mu(\a^{-1}g,\gradt\btau)
~\mbox{ from the equilibrium equation},
\eeq
where 
\beq
\a(x) = \trace(A(x))/d.
\eeq
Note that $\trace(A)$ equals to the sum of its eigenvalues. Since it is assumed that $A$ is symmetric and positive definite, the scalar function $\a(x)$ is between the minimum and maximum of eigenvalues of $A(x)$ for all $x\in \O$. 

Then we obtain the following problem: find $(\bsigma,u) \in \bX$, such that 
\beq\label{eq_bb}
B((\bsigma,u),(\btau,v))
= (\bff,\btau) + (g,v) -\k(\bff,\btau-A\nabla v) + \mu(\a^{-1}g,\gradt\btau)\quad \forall(\btau,v)\in \bX
\eeq
with the bilinear form $B$ defined as follows, for $(\bchi,w) \in \bX$ and $(\btau,v) \in \bX$,
\beq\notag
B((\bchi,w),(\btau,v)) :=(A^{-1}\bchi,\btau) + (\nabla w,\btau) +(\gradt\bchi,v)-\kappa(A^{-1}\bchi+\nabla w,\btau-A\nabla v) + \mu(\a^{-1}\gradt\bchi,\gradt\btau).
\eeq
Let $(\bsigma,u) = (\btau,v)$ in $B((\bsigma,u),(\btau,v))$, and use the fact that
\beq\label{eq_interbyparts}
(\nabla v,\btau)+(\gradt\btau,v)=0 \quad\forall(\btau,v)\in \bX,
\eeq
We get
\beq\notag
B((\btau,v),(\btau,v)) = (1-\kappa)\|A^{-1/2}\btau\|^2_0 +\mu\|\a^{-1/2}\gradt\btau\|^2_0+\k\|A^{1/2}\nabla u\|_0^2.
\eeq
To have the coercivity, $1-\kappa$ should be positive. For convenience, we let $\kappa=1/2$. Then by \eqref{eq_interbyparts}, \eqref{eq_bb} can be written as: find $(\bsigma,u) \in \bX$, such that 
\beq\label{eq_varfrom}
B_\theta((\bsigma,u),(\btau,v)) = F_\theta(\btau,v)\quad\forall (\btau,v)\in \bX,
\eeq
where, for $(\bchi,w) \in \bX$ and $(\btau,v) \in \bX$, the bilinear form $B_\theta$ and the linear form $F_\theta$ are defined as follows:
\begin{eqnarray}
\label{Btheta1}
B_\theta((\bchi,w),(\btau,v)) &=&(A^{-1}\bchi,\btau) + (A\nabla w,\nabla v) + (\nabla w, \btau ) - (\bchi, \nabla v) + (\theta\a^{-1}\gradt\bchi,\gradt\btau)\\ \label{Btheta2}
&=&(A^{-1}\bchi+\nabla w,\btau+A\nabla v)- 2(\bchi,\nabla v)+ (\theta\a^{-1}\gradt\bchi,\gradt\btau), \\ \label{Ftheta}
F_\theta(\btau,v) &=& (\bff,\btau+A\nabla v)+2(g,v)+(\theta\a^{-1}g,\gradt\btau).
\end{eqnarray}
Note that $B_\theta$ is not symmetric. We will give an equivalent symmetric version in Section \ref{sym_form}. We will discuss two choices of $\theta$, $\theta =1$ and a mesh dependent $\theta$, such that $\theta|_K = h^2_K$ for all $K\in\cT$. 

\begin{rem} In fact, we have a third case that $\theta=0$. However, unlike the first two cases, this case has no corresponding least-squares formulation. We also cannot associate its a posteriori error estimator with a corresponding least-squares error estimator. Thus, we will not discuss this choice in the paper. Some discussions of the case $\theta=0$ can be found in \cite{MH:02,CL:08}.
\end{rem}

The formulas \eqref{Btheta1} and \eqref{Btheta2} are two equivalent ways to write the bilinear form. 
\subsection{Some analysis for the augmented mixed formulations} \label{analysis}
Define
\beq \label{thetanorm}
\tri(\btau,v)\tri_\theta:= (\|A^{1/2}\nabla v\|^2_0+\|A^{-1/2}\btau\|^2_0 +\|\sqrt{\theta/\a}\gradt\btau\|^2_0)^{1/2}\quad \forall (\btau,v)\in \bX.
\eeq
By the definition of the bilinear forms $B_\theta$ in \eqref{Btheta1}, we immediately have the coercivity:
\beq\label{ine_coe1}
B_\theta((\btau,v),(\btau,v))= \tri(\btau,v)\tri_\theta^2 \quad \forall (\btau,v)\in \bX.
\eeq
It is also easy to derive the continuity of $B_\theta$: 
\begin{eqnarray}  \nonumber
B_\theta((\bchi,w), (\btau,v)) &\leq&\|A^{-1/2}\bchi\|_0 \|A^{-1/2}\btau\|_0 + \|A^{1/2}\nabla w\|_0 \|A^{1/2}\nabla v\|_0 + \|A^{1/2}\nabla w\|_0 \|A^{-1/2}\btau\|_0  \\\nonumber
&& + \|A^{-1/2}\bchi\|_0 \|A^{1/2}\nabla v\|_0  +\|\sqrt{\theta/\a}\gradt\bchi\|_0\|\sqrt{\theta/\a}\gradt\btau\|_0 \\\nonumber
&=& (\|A^{-1/2}\bchi\|_0+\|A^{1/2}\nabla w\|_0) (\|A^{-1/2}\btau\|_0+\|A^{1/2}\nabla v\|_0)+\|\sqrt{\theta/\a}\gradt\bchi\|_0\|\sqrt{\theta/\a}\gradt\btau\|_0 \\ \label{ine_conA1} 
&\leq& 2 \tri(\bchi,w)\tri_\theta\tri(\btau,v)\tri_\theta,\quad \forall (\bchi,w),(\btau,v)\in \bX
\end{eqnarray}
With the coercivity and continuity of the bilinear form $B_\theta$, by the Lax-Milgram Lemma,
\eqref{eq_varfrom} has a unique solution $(\bsigma,u)\in \bX$.

Let $\Sigma_{h,N}\subset H_N(\divvr;\O)$ and $V_{h,D}\subset H^1_D(\O)$ be two finite dimensional subspaces, then we have the following discrete problem:
find $(\bsigma_{h},u_{h}) \in \Sigma_{h,N}\times V_{h,D}$ such that
\beq \label{disweak}
B_\theta((\bsigma_{h},u_{h}),(\btau_h,v_h))=F_\theta(\btau_h,v_h), \quad \forall (\btau_h,v_h)\in \Sigma_{h,N}\times V_{h,D}.
\eeq
Since $\Sigma_{h,N}\times V_{h,D} \subset \bX$, we have the well-posedness of the discrete problems \eqref{disweak}. Also, the following Galerkin orthogonality is true:
\beq\label{eq_ortho}
B_\theta((\bsigma-\bsigma_{h},u-u_{h}),(\btau_h,v_h))=0, \quad \forall (\btau_h,v_h)\in \Sigma_{h,N}\times V_{h,D}.
\eeq
The following Cea's lemma type of best approximation property is also true.
\begin{thm}\label{a_priori}
Assume that $(\bsigma_{h},u_{h})$ is the solution of problem \eqref{disweak}, and $(\bsigma,u)$ is the solution of the problem \eqref{eq_Darcy}. The following best-approximation result is true:
\begin{eqnarray}\label{apriori}
\tri(\bsigma-\bsigma_{h},u-u_{h})\tri_\theta &\leq& 2\inf_{(\btau_h,v_h)\in \Sigma_{h,N}\times V_{h,D}}\tri(\bsigma-\btau_h,u-v_h)\tri_\theta.
\end{eqnarray}
\end{thm}
\begin{proof}
Let $(\btau_h,v_h)\in \Sigma_{h,N}\times V_{h,D}$. Using the coercivity and the continuity of the bilinear form, the Galerkin orthogonality, and the Cauchy-Schwarz inequality, we have
\begin{eqnarray*}
\tri(\bsigma-\bsigma_{h},u-u_{h})\tri^2_\theta &=& B_\theta((\bsigma-\bsigma_{h},u-u_{h}),(\bsigma-\bsigma_{h},u-u_{h})) =B_\theta((\bsigma-\bsigma_{h},u-u_{h}),(\bsigma-\btau_{h},u-v_{h}))\\
&\leq&2\tri(\bsigma-\bsigma_{h},u-u_{h})\tri_\theta \tri(\bsigma-\btau_h,u-v_h)\tri_\theta,
\end{eqnarray*}
which implies \eqref{apriori}.
\end{proof}
To derive the a posteriori error estimate, we need the following lemma.
\begin{lem}\label{errorrep} (Error representation) 
Let $(\bsigma,u)$ be the solution of \eqref{eq_Darcy},  $(\bsigma_{h},u_{h})$ be the solution of \eqref{disweak}, and $v_h\in  V_{h,D}$ be an arbitrary function in the discrete space $V_{h,D}$. We have the following error representation with $\bE = \bsigma-\bsigma_{h}$ and $e= u-u_{h}$:
\beq \label{error_rep}
\tri(\bE,e)\tri^2_\theta =  (\bff-A^{-1}\bsigma_{h}-\nabla u_{h},\bE+A\nabla (e- v_h))+(\theta \a^{-1}(g-\gradt\bsigma_{h}),\gradt\bE) +2(g-\gradt\bsigma_{h},e- v_h).
\eeq
\end{lem}
\begin{proof}
Let $ \tilde{e}= e- v_h$.
By the fact $\bE \in H_N(\divvr;\O)$ and $e\in H_D^1(\O)$,  the coercivity \eqref{ine_coe1}, the Galerkin orthogonality 
\beq \label{GO}
B_\theta((\bE,e),(0,v_h)) = 0, \quad \forall v_h\in  V_{h,D},
\eeq 
the definitions of $B_\theta$ \eqref{Btheta1} and $F_\theta$ \eqref{Ftheta}, and the integration by parts, we get
\begin{eqnarray*}
\tri(\bE,e)\tri^2_\theta &=& B_\theta((\bE,e),(\bE,e))=B_\theta((\bE,e),(\bE,\tilde{e}))
= F_\theta(\bE,\tilde{e}) - B_\theta((\bsigma_{h},u_{h}),(\bE,\tilde{e}))\\
&=& (\bff-A^{-1}\bsigma_{h}-\nabla u_{h},\bE+A\nabla \tilde{e})+(\theta \a^{-1}(g-\gradt\bsigma_{h}),\gradt\bE)+2(g,\tilde{e}) +2(\bsigma_h,\nabla \tilde{e}) \\
&=& (\bff-A^{-1}\bsigma_{h}-\nabla u_{h},\bE+A\nabla \tilde{e})+(\theta \a^{-1}(g-\gradt\bsigma_{h}),\gradt\bE) +2(g-\gradt\bsigma_{h},\tilde{e}).
\end{eqnarray*}
The lemma is then proved.
\end{proof}

\subsection{Symmetric formulations} \label{sym_form}
The formulation \eqref{eq_varfrom} is non-symmetric. In many situations, for example, eigenvalues problems or developing efficient linear solvers, symmetric formulations are always preferred. Also, we can always associate a Ritz-minimization variational principle to a symmetric problem. Luckily, the method  \eqref{eq_varfrom}  is equivalent to a symmetric GLS formulation by adding least-squares residuals
$$
-\dfrac{1}{2}(A^{-1}\bsigma + \nabla u, \btau + A\nabla v) = -\dfrac{1}{2}(\bff,\btau+A\nabla v)
\quad\mbox{and}\quad
\dfrac{1}{2}(\a^{-1}\gradt\bsigma,\gradt\btau) = \dfrac{1}{2}(\a^{-1}g,\gradt\btau)
$$
to the following symmetric saddle point mixed formulation: find $(\bsigma, u) \in\bX$, such that
$$
(A^{-1}\bsigma, \btau)-(\gradt \btau, u)-(\gradt\bsigma, v)= (\bff, \btau)-(g, v), ~\forall (\btau, v) \in\bX.
$$
We have the following symmetric formulation: find $(\bsigma, u) \in\bX$, such that
\beq \label{eq_varfrom_sym}
B_{sym,\theta}((\bsigma, u),(\btau,v))= F_{sym,\theta}(\btau,v), \quad \forall (\btau, v) \in\bX,
\eeq
with the forms are defined for $(\bchi,w) \in \bX$ and $(\btau,v) \in \bX$,
\begin{eqnarray}
\label{Btheta_eym}
B_{sym,\theta}((\bchi,w),(\btau,v)) &:=&(A^{-1}\bchi,\btau)  + (\nabla w, \btau ) + (\bchi, \nabla v) -(A\nabla w,\nabla v) + (\theta\a^{-1}\gradt\bchi,\gradt\btau).\\ 
\label{Ftheta_eym}
F_{sym,\theta}(\btau,v) &:=& (\bff,\btau-A\nabla v)-2(g,v)+(\theta\a^{-1}g,\gradt\btau).
\end{eqnarray}
Let $\Sigma_{h,N}\subset H_N(\divvr;\O)$ and $V_{h,D}\subset H^1_D(\O)$ be two finite dimensional subspaces, then we have the following discrete problem corresponding to \eqref{eq_varfrom_sym}:
find $(\bsigma_{h},u_{h}) \in \Sigma_{h,N}\times V_{h,D}$ such that
\beq \label{disweak_sym}
B_{sym,\theta}((\bsigma_{h},u_{h}),(\btau_h,v_h))=F_{sum,\theta}(\btau_h,v_h), \quad \forall (\btau_h,v_h)\in \Sigma_{h,N}\times V_{h,D}.
\eeq

It is easy to see that \eqref{eq_varfrom} and  \eqref{eq_varfrom_sym} are equivalent since replacing the test function $v$ by $-v$ in one formulation leads to the other formulation. By doing so, we know that \eqref{eq_varfrom} and  \eqref{eq_varfrom_sym} and their corresponding finite element formulations \eqref{disweak} and \eqref{disweak_sym} produce identical solutions. Thus all the analysis of the non-symmetric formulations can be applied to the symmetric versions.

Alternatively, we can establish the inf-sup stability of the symmetric formulation directly.

\begin{lem}
The following robust inf-sup stability with the stability constant being $1$ holds:
\beq \label{infsup}
\inf_{(\bchi,w)\in \bX}\sup_{(\btau,v)\in \bX} \dfrac{B_{sym,\theta}((\bchi,w),(\btau,v))}{\tri (\bchi,w) \tri_{\theta}  \tri (\btau,v) \tri_{\theta} }
=
\inf_{(\btau,v)\in \bX}\sup_{(\bchi,w)\in \bX} \dfrac{B_{sym,\theta}((\bchi,w),(\btau,v))}{\tri (\bchi,w) \tri_{\theta}  \tri (\btau,v) \tri_{\theta} }
 \geq 1.
\eeq
\end{lem}
\begin{proof}
Due to the fact $B_{sym,\theta}$ is symmetric, we only need to show one  inf-sup condition in \eqref{infsup}.
Let $(\btau,v) =  (\bchi,-w)$ and using the facts that $B_{\theta}((\bchi,w),(\bchi,w) = B_{sym,\theta}((\bchi,w),(\bchi,-w))$ and \eqref{ine_coe1} then 
\begin{eqnarray*}
\sup_{(\btau,v)\in \bX} \dfrac{B_{sym,\theta}((\bchi,w),(\btau,v))}{\tri (\btau,v) \tri_{\theta} } &\geq &\dfrac{B_{sym,\theta}((\bchi,w),(\bchi,-w))}{\tri (\bchi,w) \tri_{\theta} } 
= \dfrac{B_{\theta}((\bchi,w),(\bchi,w))}{\tri (\bchi,w) \tri_{\theta} } =  \dfrac{\tri (\bchi,w) \tri_{\theta}^2}{\tri (\bchi,w) \tri_{\theta} } = \tri (\bchi,w) \tri_{\theta}.
\end{eqnarray*}
\end{proof}
Similarly, we have the following robust discrete inf-sup stability with $\bX_h = \Sigma_{h,N}\times V_{h,D}$:
\beq
\inf_{(\bchi_h,w_h)\in \bX_h} \sup_{(\btau_h,v_h)\in \bX_h} \dfrac{B_{sym,\theta}((\bchi_h,w_h),(\btau_h,v_h))}{\tri (\bchi_h,w_h) \tri_{\theta} \tri (\btau_h,v_h) \tri_{\theta} } 
=\inf_{(\btau_h,v_h)\in\bX_h} \sup_{(\bchi_h,w_h)\in \bX_h} \dfrac{B_{sym,\theta}((\bchi_h,w_h),(\btau_h,v_h))}{\tri (\bchi_h,w_h) \tri_{\theta} \tri (\btau_h,v_h) \tri_{\theta} } 
\geq 1
.
\eeq
Also, it is easy to show the continuity of $B_{sym,\theta}$ as in \eqref{ine_conA1}:
\beq  \label{ine_conA1_sym} 
B_{sym,\theta}((\bchi,w), (\btau,v))
\leq 2 \tri(\bchi,w)\tri_\theta\tri(\btau,v)\tri_\theta,\quad \forall (\bchi,w),(\btau,v)\in \bX.
\eeq
Using the theorem in \cite{XZ:03}, we immediately have following robust best approximation theorem, which is identical to Theorem \ref{a_priori}.
\begin{thm}%\label{a_priori_mixed1}
Assume that $(\bsigma_{h},u_{h})$ is the solution of problem \eqref{disweak_sym}, and $(\bsigma,u)$ is the solution of the problem \eqref{eq_Darcy}. The following robust best-approximation result is true:
\begin{eqnarray}\label{apriori_sym}
\tri(\bsigma-\bsigma_{h},u-u_{h})\tri_\theta &\leq& 2\inf_{(\btau_h,v_h)\in \Sigma_{h,N}\times V_{h,D}}\tri(\bsigma-\btau_h,u-v_h)\tri_\theta.
\end{eqnarray}
\end{thm}

\begin{rem}
The formulation \eqref{eq_varfrom_sym} (with $\theta=1$ and $0$) can be found in Section 4.1 (Symmetric stabilizations in $H(\divvr;\O)\times H^1(\O)$) of \cite{CL:08}. The paper \cite{CL:08} mainly wanted to use continuous finite elements to approximate both $\bsigma$ and $u$. It did mention the possible usage of $H(\divvr)$-conforming finite elements. Robust a priori error estimate with respect to the coefficient matrix $A$  is not sought in \cite{CL:08}. A posteriori error estimate is not discussed in \cite{CL:08}.
\end{rem}

\section{Assumptions on the coefficient matrix $A$ and Robust Interpolations}
\setcounter{equation}{0}
In the remaining part of the paper, for simplicity of the presentation, we assume the following assumption on $A$: 
\begin{assumption} \label{asmp_A} {\bf Piecewise constant assumption on $A$.}
We assume that $A=\a(x)I$ where $\a(x)$ is positive and piecewise constant function in $\O$ with possible large jumps across subdomain boundaries (interfaces):
$$
\a(x)=\a_i >0 \mbox{ in } \O_i
$$
for $i = 1, \cdots, n$. Here, $\{\O_i\}_{i=1}^n$ is a partition of the domain $\O$ with $\O_i$ being an open polygonal domain. 
\end{assumption}

\begin{rem} For the more general cases of $A$, the analysis in this paper will still be valid. However, the genetic constants appeared in the paper will depend on the ratio $\lambda_{\max,K}/\lambda_{\min,K}$, for all $K\in\cT$, where $\lambda_{\max,K}$ and $\lambda_{\min,K}$ are the respective maximal and minimal eigenvalues of $A_K := A|_K$. See discussion in \cite{BeVe:00,CHZ:21}. 
\end{rem}

We then discuss the quasi-monotonicity assumption and robust quasi-interpolation based on this assumption.
\begin{assumption} \label{asmp_QMA} {\bf Quasi-monotonicity assumption (QMA).} Assume that any two different subdomains $\overline{\O}_i$ and  $\overline{\O}_i$, which share at least one point, have a connected path passing from  $\overline{\O}_i$ to  $\overline{\O}_j$ through adjacent subdomains such that the diffusion coefficient $\a(x)$ is monotone along this path.
\end{assumption}
It is also common to use Cl\'{e}ment-type interpolation operators (see, e.g., \cite{BeVe:00,Pet:02}) for establishing the reliability bound of a posteriori error estimators. Following \cite{BeVe:00}, one can define the interpolation operator $I_{rcl}: L^2(\O)\rightarrow S^1_D$ (see \cite{CZ:09}  for more details) so that the following estimates are true under Assumption \ref{asmp_QMA} (QMA):
\beq \label{rcl}
 \|\a_K^{\frac12}(v-I_{rcl} v)\|_{0,K} +h_K\|\a_K^{\frac12}\nabla (v-I_{rcl} v)\|_{0,K} \leq C\,h_K \| \a^{\frac12}\nabla v\|_{0,\Delta_K} \quad \forall K\in\cT, v\in H^1_D(\O),
\eeq
where $\Delta_K$ is the union of all elements that share at least one vertex with K,

%\subsubsection{Poincar\'e-Friedrichs inequality}

In general, we do not have the following robust Poincar\'e-Friedrichs inequality
\beq \label{robust_PI}
\|\a^{1/2} v\|_0 \leq C \|\a^{1/2}\nabla v\|_0, \quad v\in H^1_D(\O),
\eeq
where $C>0$ is independent of $\a$.

For the following special case, \eqref{robust_PI} does holds.

\begin{lem} \label{Poincar\'e}
Assume that each $\O_i$ in the Assumption \ref{asmp_A} has a part of the Dirichlet boundary condition with a positive measure, i.e., 
$$
\mbox{measure}\{ \Gamma_D^i \} >0, \mbox{ where } \Gamma_D^i = \p\O_i\cap\Gamma_D  \quad \mbox{for } i =1,\cdots n,
$$ 
then the robust Poincar\'e inequality \eqref{robust_PI} is true with $C>0$ being independent of $\a$.
\end{lem} 
The proof is strghtforward since $v_i=v|_{\O_i} \in \{ w \in H^1(\O_i), w|_{\Gamma_D^i} =0\} $, and $\a$ in $\O_i$ is $\a_i $, a single constant, we have 
$$
\|\a_i^{1/2} v_i\|_{0,\O_i} \leq C \|\a_i^{1/2}\nabla v_i\|_{0,\O_i}, 
$$
with $C>0$ independent of $\a_i$. Summing up all the subdomains we get the robust Poincar\'e-Friedrichs inequality for this special case.

%We will briefly discuss the general case $A$ in section . 

\section{The first augmented mixed formulation: $\theta = 1$}
\setcounter{equation}{0}
%\subsection{The variational problem of the first augmented mixed formulation}
In this section, we consider that case that $\theta = 1$.
%From \eqref{eq_varfrom}, for $(\bchi,w) \in \bX$ and $(\btau,v) \in \bX$, define
%\begin{eqnarray}
%A_1((\bchi,w),(\btau,v))
%&:=& (\a^{-1}\bchi,\btau) +(\a\nabla w,\nabla v)+(\nabla w,\btau) -(\bchi,\nabla v)
%+(\a^{-1}\gradt\bchi,\gradt\btau),\\
%F_1(\btau,v) &:=& (\bff,\btau)+2(g,v)+(\bff,\a\nabla v)+(\a^{-1}g,\gradt\btau).
%\end{eqnarray}
%We have the first augmented mixed formulation:
%Find $(\bsigma,u)\in \bX$ such that
%\beq\label{con var formula1}
%A_1((\bsigma,u),(\btau,v)) = F_1(\btau,v), \quad \forall (\btau,v)\in \bX.
%\eeq
%Let
%\beq
%\tri(\btau,v)\tri_1:= (\|A^{1/2}\nabla v\|^2_0+\|A^{-1/2}\btau\|^2_0 +\|\a^{-1/2}\gradt\btau\|^2_0)^{1/2}\quad \forall (\btau,v)\in \bX.
%\eeq
%By the definition of the bilinear forms $A_1$, we immediately have the coercivity:
%\beq\label{ine_coe1}
%A_1((\btau,v),(\btau,v))= \tri(\btau,v)\tri_1^2 \quad \forall (\btau,v)\in \bX.
%\eeq
%
%It is easy to derive the continuity of $A_1$: 
%\begin{equation}\label{ine_conA1}
%A_1( (\bchi,w), (\btau,v) )\leq 2\tri(\bchi,w)\tri_1\tri(\btau,v)\tri_1,\quad \forall (\bchi,w),(\btau,v)\in \bX.
%\end{equation}
%With the coercivity and continuity of the bilinear form $A_1$, by the Lax-Milgram Lemma,
%\eqref{con var formula1} has a unique solution $(\bsigma,u)\in \bX$.
%
%\subsection{The first augmented mixed FEM and the a priori estimate}
%
For simplicity, we consider the finite element approximation in $RT_{0,N}\times S_{1,D}\subset \bX$. The discrete problem is: find $(\bsigma_{1,h},u_{1,h}) \in RT_{0,N}\times S_{1,D}$ such that
\beq\label{disweak1}
B_1((\bsigma_{1,h},u_{1,h}),(\btau_h,v_h))=F_1(\btau_h,v_h), \quad \forall (\btau_h,v_h)\in RT_{0,N}\times S_{1,D},
\eeq
where $B_1$ and $F_1$ are the corresponding forms \eqref{Btheta1} and \eqref{Ftheta} with $\theta=1$.
Based on the discussions in Section \ref{analysis}, we have the well-posedness of discrete problem \eqref{disweak1}. Let $\tri(\cdot,\cdot)\tri_1$ be the norm defined in \eqref{thetanorm} with $\theta=1$. We have the following locally robust and optimal a priori error estimate.
\begin{thm}\label{a_priori_mixed1}
Let $(\bsigma,u)$ be the solution of \eqref{eq_Darcy} and $(\bsigma_{1,h},u_{1,h})$ be the solution of problem \eqref{disweak1}, respectively. We have the following a priori estimate:
\begin{eqnarray}\label{apriori1}
\tri(\bsigma-\bsigma_{1,h},u-u_{1,h})\tri_1 &\leq& 2\inf_{(\btau_h,v_h)\in RT_{0,N}\times S_{1,D}}\tri(\bsigma-\btau_h,u-v_h)\tri_1.
\end{eqnarray}
Under Assumption \ref{asmp_A} on the coefficients, if we further assume that $u|_K\in H^{1+s_K}(K)$, $(\nabla u-\bff)|_K\in H^{q_K}(K)^d$, $g|_K\in H^{t_K}(K)$ for $K\in\cT$, where the local regularity indexes $s_K$, $q_K$, and $t_K$ satisfy the following assumptions: $0<s_K\leq 1$ in two dimensions and $1/2<s_K\leq 1$ in there dimensions,  $1/2<q_K\leq 1$ with the constant $C_{rt}>0$ being unbounded as $q_K\downarrow 1/2$, and $0<t_K\leq 1$, then the following local robust and local optimal a priori error estimate holds: there exists a constant $C$ independent of $\a$ and the mesh-size, such that
\beq \label{a_priori_mixed_1}
\tri(\bsigma-\bsigma_{1,h},u-u_{1,h})\tri_1 \leq C\sum_{K\in\cT}\a_K^{1/2}\left(h_K^{s_K}|\nabla u|_{s_K,K}+ C_{rt} h_K^{q_K}|\nabla u-\bff|_{q_K,K} + \a_K^{-1}h_K^{t_K}|g|_{t_K,K}
\right).
\eeq
\end{thm}
\begin{proof}
The result \eqref{apriori1} is from \eqref{apriori}.
The result of \eqref{a_priori_mixed_1} can be derived from regularity assumptions and interpolation results \eqref{nodal_inter}, \eqref{RT_inter1}, and \eqref{RT_inter2}.
\end{proof}
\begin{rem} In theorem \ref{a_priori_mixed1}, when $\bff=0$, then $s_K = q_K$ for each $K \in \cT$. When $\bff\neq \bzero$, the local regularity of $\bsigma$ and $\nabla u$ can be different, and the regularity of $g$ (which is $\gradt\bsigma$) is also independent of that of $u$.
\end{rem}

\subsection{A least-squares a posteriori error estimator for the first augmented mixed formulation}

We discuss a posteriori error estimator for the first augmented mixed formulation in this subsection.

Define an indicator on each $K\in\cT$:
\begin{eqnarray*}
\eta_{1,K}(\bsigma_{1,h},u_{1,h}) &=& \Big\{\|\a^{-1/2}(g-\gradt\bsigma_{1,h})\|^2_{0,K}
+
\|\a^{1/2}(\bff-\nabla u_{1,h}-\a^{-1}\bsigma_{1,h})\|^2_{0,K}\Big\}^{1/2}.
\end{eqnarray*}
Define the corresponding global a posteriori error estimator:
\beq\label{ap_aug1}
\eta_1(\bsigma_{1,h},u_{1,h})
= \Big\{\|\a^{-1/2}(g-\gradt\bsigma_{1,h})\|^2_{0}
+
\|\a^{1/2}(\bff-\nabla u_{1,h}-\a^{-1}\bsigma_{1,h})\|^2_{0}\Big\}^{1/2}.
\eeq
Note that the error estimator $\eta_1(\bsigma_{1,h},u_{1,h})$ is actually a least-squares error estimator, see discussion below in section \ref{comparison_1}.
\begin{thm}\label{posteriori1}
Let $(\bsigma,u)$ be the solution of \eqref{eq_Darcy} and $(\bsigma_{1,h},u_{1,h})$ be the solution of problem \eqref{disweak1}, repectively. Assume that Assumption \ref{asmp_A} on the coefficients and Assumption \ref{asmp_QMA} (QMA) are true, then there exists positive constants $C$ independent of $\a$ and the mesh-size such that the following reliability holds:
\beq \label{rel_eta1}
\tri(\bsigma-\bsigma_{1,h}, u-u_{1,h})\tri_1 \leq C\eta_1(\bsigma_{1,h},u_{1,h}).
\eeq
\end{thm}

\begin{proof}
Let $\bE_1 = \bsigma-\bsigma_{1,h}$ and $e_1= u-u_{1,h}$.
In \eqref{error_rep}, let $\theta =1$ and $v_h = I_{rcl}e_1$, we have
$$
\tri(\bE_1,e_1)\tri^2_1 = (\bff-\a^{-1}\bsigma_{1,h}-\nabla u_{1,h},\bE_1+\a\nabla (e_1- I_{rcl} e_1))+(g-\gradt\bsigma_{1,h},\a^{-1}\gradt\bE_1+2(e_1- I_{rcl} e_1)).
$$
Applying Cauchy-Schwarz and triangle inequalities, we get
\begin{eqnarray} \label{EE}
\tri(\bE_1,e_1)\tri^2_1 &\leq& \|\a^{1/2}(\bff-\a^{-1}\bsigma_{1,h}-\nabla u_{1,h})\|_0(\|\a^{-1/2}\bE_1\|_{0}+\|\a^{1/2}\nabla (e_1- I_{rcl} e_1)\|_0)\\ \nonumber
&&+\|\a^{-1/2}(g-\gradt\bsigma_{1,h})\|_{0} (\|\a^{-1/2}\gradt\bE_1\|_0 +2\|\a^{1/2}(e_1- I_{rcl} e_1)\|_0.
\end{eqnarray}
By the robust Cl\'{e}ment interpolation result \eqref{rcl} under Assumption \ref{asmp_QMA} and the fact that the mesh-size is bounded, we have the following robust results,
$$
\|\a^{1/2}\nabla (e_1- I_{rcl} e_1)\|_0 \leq C \|\a^{1/2}\nabla e_1\|_0 \quad\mbox{and}\quad
\|\a^{1/2}(e_1- I_{rcl} e_1)\|_0 \leq C \|\a^{1/2}h\nabla e_1\|_0 \leq C \|\a^{1/2}\nabla e_1\|_0.
$$ 
Substitute these two robust results into \eqref{EE}, we get
\begin{eqnarray*} 
\tri(\bE_1,e_1)\tri^2_1 &\leq& C (\|\a^{1/2}(\bff-\a^{-1}\bsigma_{1,h}-\nabla u_{1,h})\|_0+\|\a^{-1/2}(g-\gradt\bsigma_{1,h})\|_{0}) \\
&&(\|\a^{-1/2}\bE_1\|_{0}+\|\a^{-1/2}\gradt\bE_1\|_{0}+\|\a^{1/2}\nabla e_1\|_0)\\ 
&\leq&C\eta_1(\bsigma_{1,h}, u_{1,h})\tri(\bE_1,e_1)\tri_1 .
\end{eqnarray*}
The robust reliability result \eqref{rel_eta1} is proved. 
%\begin{eqnarray*}
%&\leq& \|\a^{1/2}(\bff-\a^{-1}\bsigma_{1,h}-\nabla u_{1,h})\|_0(\|\a^{-1/2}\bE_1\|_{0}+C\|\a^{1/2}\nabla e_1\|_0)\\
%&&+\|\a^{-1/2}(g-\gradt\bsigma_{1,h})\|_{0} (\|\a^{-1/2}\gradt\bE_1\|_0 +2C\|h\a^{1/2}\nabla e_1\|_0
%\end{eqnarray*}
%Using the fact that the mesh size $h_K$ is bounded, we then have the result of the theorem.
\end{proof}
The efficiency of the proposed error indicator is the same as the standard least-squares a posteriori error estimator.
\begin{thm}\label{efficiency}
Let $(\bsigma,u)$ be the solution of \eqref{eq_Darcy} and $(\bsigma_{1,h},u_{1,h})$ be the solution of problem \eqref{disweak1}, respectively. Assume that Assumption \ref{asmp_A} on the coefficients is true, we have the following efficiency:
\beq\label{eff_eta1u1}
\eta_{1,K}(\bsigma_{1,h},u_{1,h}) \leq \sqrt{2} \tri(\bsigma-\bsigma_{1,h},u-u_{1,h})\tri_{1,K}
\quad \forall K\in\cT.
\eeq
\end{thm}
\begin{proof}
By the triangle inequality and the first-order system \eqref{eq_Darcy}, for any $K\in\cT$, we have
\begin{eqnarray*}
\eta_{1,K}^2(\bsigma_{1,h},u_{1,h}) &=& \|\a^{-1/2}(g-\gradt\bsigma_{1,h})\|^2_{0,K} 
	+ \|\a^{1/2}\bff-\a^{-1/2}\bsigma_{1,h}-\a^{1/2}\nabla u_{1,h}\|^2_{0,K}\\
&=& \|\a^{-1/2}(\gradt\bsigma-\gradt\bsigma_{1,h})\|^2_{0,K} + \|\a^{-1/2}(\bsigma-\bsigma_{1,h})+\a^{1/2}\nabla (u- u_{1,h})\|^2_{0,K} \\
&\leq& 
\|\a^{-1/2}\gradt(\bsigma-\bsigma_{1,h})\|^2_{0,K} + 2( \|\a^{-1/2}(\bsigma-\bsigma_{1,h})\|^2_{0,K}+ \|\a^{1/2}\nabla (u- u_{1,h})\|^2_{0,K}  \big)\\
&=& 2 \tri(\bsigma-\bsigma_{1,h},u-u_{1,h})\tri^2_{1,K}.
\end{eqnarray*}
The theorem is proved.
\end{proof}
\begin{rem} Note that  Assumption \ref{asmp_QMA} (QMA) is not required for the robustness of the efficiency bound.
\end{rem}

\subsection{Comparison with the $L^2$-based LSFEM} \label{comparison_1}
For the first-order system \eqref{eq_Darcy}, the $L^2$-based least-squares functional is
\beq \label{lsfunctional}
J(\btau,v;\bff,g) := \|\a^{1/2}\nabla v +\a^{-1/2}\btau - \a^{1/2}\bff\|_0^2+  \|\a^{-1/2}(\gradt\btau - g)\|_0^2, \quad (\btau,v)\in \bX.
\eeq
Then the $L^2$-based least-squares minimization problem is: find $(\bsigma,u) \in \bX$, such that 
\beq
J(\bsigma,u;\bff,g) = \inf_{(\btau,v)\in \bX} J(\btau,v;\bff,g).
\eeq 
Equivalently, it can be written in a weak form as: find $(\bsigma,u) \in \bX$, such that 
\beq \label{lsbform}
b((\bsigma,u), (\btau,v)) = (\bff,\btau+\a\nabla v)+ (\a^{-1}g,\gradt\btau), \quad \forall (\btau,v) \in \bX,
\eeq
where the least-squares bilinear form $b$ is defined as follows:
\beq
b((\bchi,w), (\btau,v)) = (\a^{-1}\bchi+\nabla w,\btau+\a\nabla v)+ (\a^{-1}\gradt\bchi,\gradt\btau), \quad \forall (\bchi,w), (\btau,v) \in \bX.
\eeq
The lowest-order least-squares finite element method (LSFEM) is: find $(\bsigma_{h}^{ls},u_h^{ls}) \in RT_{0,N}\times S_{1,D}$, such that
\beq \label{lsfem1}
J(\bsigma^{ls}_h,u^{ls}_h;\bff,g) = \inf_{(\btau,v)\in RT_{0,N}\times S_{1,D}} J(\btau,v;\bff,g).
\eeq 
Or, equivalently, find $(\bsigma_{h}^{ls},u_h^{ls}) \in RT_{0,N}\times S_{1,D}$, such that,
\beq  \label{lsfem1h}
b((\bsigma_{h}^{ls},u_h^{ls}), (\btau,v)) = (\bff,\btau+\a\nabla v)+ (\a^{-1}g,\gradt\btau), \quad \forall (\btau,v) \in RT_{0,N}\times S_{1,D}.
\eeq
The key ingredient to establish  a priori and a posteriori error estimates of the LSFEM \eqref{lsfem1} or \eqref{lsfem1h} is the following norm equivalence:
\beq \label{ls_equvalience}
C_{coe} \tri(\btau,v)\tri_1 ^2 \leq J(\btau,v;0,0) \leq C_{con} \tri(\btau,v)\tri_1 ^2, \quad \forall (\btau,v)\in \bX,
\eeq
for some positive constants $C_{coe}$ and $C_{con}$. The first inequality of \eqref{ls_equvalience} is the coercivity of the least-squares bilinear form \eqref{lsbform}:
\beq \label{ls_coe}
b((\btau,v), (\btau,v)) \geq C_{coe} \tri(\btau,v)\tri_1^2\quad \forall (\btau,v)\in \bX.
\eeq
The second inequality  of \eqref{ls_equvalience} is equivalent to the continuity of the least-squares bilinear form \eqref{lsbform} since it is symmetric:
\beq
b((\bchi,w), (\btau,v)) \leq C_{con} \tri(\bchi,w)\tri_1\tri(\btau,v)\tri_1,\quad \forall (\bchi,w), (\btau,v) \in \bX.
\eeq
With the coercivity and continuity, we can easily get the a priori error estimate of the LSFEM \eqref{lsfem1h}:
\beq \label{ls_apriori}
\tri(\bsigma-\bsigma_{h}^{ls},u-u_{h}^{ls})\tri_1 \leq \dfrac{C_{con}}{C_{coe}}\inf_{(\btau_h,v_h)\in RT_{0,N}\times S_{1,D}}\tri(\bsigma-\btau_h,u-v_h)\tri_1.
\eeq 
By the triangle inequality, it is easy to prove that continuity constant $C_{con}$ is a constant independent of $\a$ (actually, $C_{con}=2$ for this case, see arguments of the proof in Theorem \ref{efficiency}). On the other hand, the coercivity constant $C_{coe}$ usually depends $\a$. The reason is that the proof of the coercivity of the least-squares bilinear form \eqref{lsbform} requires the Poincar\'e inequality, which is usually not robust with respect to $\a$. For the special case discussed in Lemma \ref{Poincar\'e}, we have the following result:
\begin{thm} \label{lsfem_robust}
Assume that each $\O_i$ in the Assumption \ref{asmp_A} has a part of the Dirichlet boundary condition with a positive measure, then the coercivity constant $C_{coe}$ the norm equivalence \eqref{ls_equvalience}  and the coercivity \eqref{ls_coe} is independent of $\a$.
\end{thm}
\begin{proof}
For a $v\in H^1_D(\O)$, let $\btau$ be an arbitrary vector in $H_N(\divvr;\O)$, then
\begin{eqnarray*}
\|\a^{1/2}\nabla v\|_0^2 &=& (\a\nabla v,\nabla v) = (\a\nabla v+\btau,\nabla v) - (\btau,\nabla v) = (\a\nabla v+\btau,\nabla v) + (\gradt\btau, v) \\
& \leq &\|\a^{1/2}\nabla v+\a^{-1/2}\btau\|_0 \|\a^{1/2}\nabla v\|_0 + \|\a^{-1/2}\gradt\btau\|_0 \|\a^{1/2}v\|_0.
\end{eqnarray*}
By lemma \ref{Poincar\'e}, for the special setting of the theorem, a robust Poincar\'e inequality $\|\a^{1/2}v\|_0\leq C\|\a^{1/2}\nabla v\|_0$ for $v\in H^1_D(\O)$ is true. Thus, we have 
\beq
\|\a^{1/2}\nabla v\|_0 \leq \|\a^{1/2}\nabla v+\a^{-1/2}\btau\|_0 +C\|\a^{-1/2}\gradt\btau\|_0, \quad \forall (\btau,v)\in \bX, 
\eeq
with the constant $C$ independent of $\a$. It is also simple to see that, for $(\btau,v)\in \bX$,
\beq
\|\a^{-1/2}\btau\|_0  \leq \|\a^{1/2}\nabla v+\a^{-1/2}\btau\|_0 +\|\a^{1/2}\nabla v\|_0 \leq 2\|\a^{1/2}\nabla v+\a^{-1/2}\btau\|_0 +C\|\a^{-1/2}\gradt\btau\|_0.
\eeq
Thus, we prove that robust coercivity.
\end{proof}
From the above proof, we also confirm that the weight $\a^{-1/2}$ in $\|\a^{-1/2}\gradt \btau\|_0$ in \eqref{lsfunctional} is the right choice. 

From \eqref{ls_apriori}, except in the special cases where the robust version of Poincar\'e inequality holds, the a priori error estimate of the LSFEM \eqref{lsfem1}, \eqref{lsfem1h} in general is not robust with respect to $\a$. 

Let $(\bsigma_{a},u_{a}) \in \bX$, we can define the following least-squares based a posteriori error estimator:
\beq\label{ap_lsfem}
\eta_{ls}(\bsigma_{a},u_{a})
= \Big (\|\a^{-1/2}(g-\gradt\bsigma_{a})\|^2_{0} + \|\a^{1/2}(\bff-\nabla u_{a}-\a^{-1}\bsigma_{a})\|^2_{0}\Big )^{1/2} = J(\bsigma_{a},u_{a};\bff,g)^{1/2}.
\eeq
Let $\bE_a = \bsigma - \bsigma_{a}$ and $e_a = u-u_a$. Using the facts that $\bff=\nabla u +\a^{-1}\bsigma$ and $g=\gradt\bsigma$ from \eqref{eq_Darcy}, we have the following identity:
\begin{eqnarray*}
J(\bsigma_a,u_a;\bff,g) &=& \|\a^{-1/2}(g-\gradt\bsigma_{a})\|^2_{0} + \|\a^{1/2}(\bff-\nabla u_{a}-\a^{-1}\bsigma_{a})\|^2_{0}\\  
&=& \|\a^{-1/2}(\gradt\bsigma-\gradt\bsigma_{a})\|^2_{0} + \|\a^{1/2}(\nabla u +\a^{-1}\bsigma-\nabla u_{a}-\a^{-1}\bsigma_{a})\|^2_{0} = J(\bE_a,e_a;0,0). 
\end{eqnarray*}
By  \eqref{ls_equvalience},  the following reality and efficiency bounds are true,
\beq \label{ls_rel_eff}
C_{coe} \tri(\bE_a,e_a)\tri_1 ^2 \leq J(\bE_a,e_a;0,0) = J(\bsigma_a,u_a;\bff,g) \leq C_{con} \tri(\bE_a,e_a)\tri_1 ^2.
\eeq
An important fact of \eqref{ls_rel_eff} is that $(\bsigma_{a},u_{a}) \in \bX$ does not need to be the numerical solution of the LSFEM problem \eqref{lsfem1}. In fact, the pair can be any functions in $\bX$. Let  $(\bsigma_{h}^{ls},u_h^{ls}) \in RT_{0,N}\times S_{1,D}\subset \bX$ be the numerical solution of the LSFEM problem \eqref{lsfem1h}, we immediately have the reliability and efficiency of the least-squares error estimator for the LSFEM approximation \eqref{lsfem1h}, 
\beq \label{lsfem_error_analysis}
C_{coe} \tri(\bsigma-\bsigma_{h}^{ls},u-u_h^{ls})\tri_1 ^2 \leq  \eta_{ls}(\bsigma_{h}^{ls},u_{h}^{ls})^2= J(\bsigma_h^{ls},u_h^{ls};\bff,g) \leq C_{con} \tri(\bsigma-\bsigma_{h}^{ls},u-u_h^{ls})\tri_1 ^2.
\eeq
Of course, since $C_{coe}$ depends on $\a$, the a posteriori error estimator $\eta_{ls}(\bsigma_{h}^{ls},u_{h}^{ls})$ is not robust for the LSFEM approximation \eqref{lsfem1h}.

If the robustness of the estimator is not our goal, we can have the non-robust reliability and robust efficiency of the a posteriori error estimator $\eta_1(\bsigma_{1,h},u_{1,h})$ of the first augmented mixed method \eqref{ap_aug1} by using the fact that  $(\bsigma_{1,h},u_{1,h}) \in RT_{0,N}\times S_{1,D}\subset \bX$ and \eqref{ls_rel_eff}:
\beq \label{ls_am_rel_eff}
C_{coe} \tri(\bE_1,e_1)\tri_1 ^2 \leq \eta_1(\bsigma_{1,h},u_{1,h})^2 = \eta_{ls}(\bsigma_{1,h},u_{1,h})^2 \leq C_{con} \tri(\bE_1,e_1)\tri_1 ^2.
\eeq
This result is weaker than that of Theorem \ref{posteriori1}. 

\begin{rem}
It is interesting to see that for the first augmented mixed method \eqref{disweak1} and the LSFEM \eqref{lsfem1} or \eqref{lsfem1h}, their a posteriori error estimators are the same. However, one is robust, and the other one is not. The subtle difference is that the numerical solution in the a posteriori error estimator \eqref{ap_aug1} is obtained by the robust first augmented mixed method \eqref{disweak1}, and the Galerkin orthogonality \eqref{GO} is used in the error representation Lemma \ref{errorrep}. On the other hand, the reliability and efficiency of a general least-squares a posteriori error estimator do not require that the approximations are the numerical solutions of the corresponding LSFEM or augmented mixed method.  
\end{rem}

\begin{rem}
We take a comparison of the least-squares problem \eqref{lsbform} and the augmented problem \eqref{eq_varfrom} with $\theta=1$ with the bilinear form in the form of \eqref{Btheta2}. The first augmented mixed method can be viewed as adding a consistent term, 
$$
-2(\bsigma,\nabla v) = 2(\gradt \bsigma,v) = (g,v) \quad \forall v\in H^1_D(\O)
$$
to the least-squares problem \eqref{lsbform}. The extra term makes the new formulation lose the least-squares energy minimization principle, but $-2(\btau,\nabla v)$ cancels the cross-term in $\|A^{-1/2}\btau+A^{1/2}\nabla v\|_0^2$, thus the new augmented mixed method is robust in the energy norms.
\end{rem}

\begin{rem} \label{lsfem_mild}
We also need to mention that the non-robustness of the $L^2$-LSFEM is mild.  Take a close look at the proof of the robustness in Theorem \ref{lsfem_robust}, a robust Poincar\'e inequality for any $v\in H^1_D(\O)$ is needed. In general, the robust Poincar\'e inequality is not true. But, for the robust error analysis of LSFEM \eqref{lsfem_error_analysis}, we only need the fact 
\beq \label{error_Poincar\'e}
\|\a^{1/2} (u-u^{ls}_h)\|_0 \leq C \|\a^{1/2}\nabla (u-u^{ls}_h)\|_0,
\eeq
for a constant $C>0$ independent of $\a$.
Assuming that we have some $L^2$-error bound $\|u-u^{ls}_h\|_0 \leq C h^{r} \|\nabla (u-u^{ls}_h)\|_0$ for some regularity $r>0$ and $h$ being the maximum size of the mesh, see discussions in \cite{CK:06}, then $\|\a^{1/2} (u-u^{ls}_h)\|_0 \leq C \|\a^{1/2}\nabla (u-u^{ls}_h)\|_0$ for $C$ independent of $\a$ is possible with a small enough $h$. The situation of a non-uniform mesh and a solution with a low regularity is less clear, but realizing that \eqref{error_Poincar\'e} for the error $u-u^{ls}_h$ is what we needed is helpful to explain the mildness of non-robustness of the LSFEM.
\end{rem}

\begin{rem}
In \cite{BCG:15}, the same error estimator is proposed for the first augmented mixed method. In the proof of  \cite{BCG:15}, the same technique as the coercivity proof of the LSFEM  \eqref{ls_coe} is used. Since the Galerkin orthogonality \eqref{GO} is not used, the analysis in \cite{BCG:15} is not robust.
\end{rem}

\section{The second augmented mixed formulation: a mesh-weighted version}
\setcounter{equation}{0}
\subsection{The second augmented mixed formulation}
In this formulation, we choose $\theta$ to be a piecewisely defined function such that $\theta|_K=h^2_K$, for $K\in\cT$.

We consider the finite element approximation in two pairs: 
$RT_{0,N}\times S_{1,D}\subset \bX$ and $BDM_{1,N}\times S_{2,D}\subset \bX$. The notation $\Sigma_{h,N}\times V_{h,D}$ is used to represent these two choices.
The discrete problem is: find $(\bsigma_{2,h},u_{2,h}) \in \Sigma_{h,N}\times V_{h,D}$ such that
\beq\label{disweak2}
B_2((\bsigma_{2,h},u_{2,h}),(\btau_h,v_h))=F_2(\btau_h,v_h), \quad \forall (\btau_h,v_h)\in \Sigma_{h,N}\times V_{h,D},
\eeq
where $B_2$ and $F_2$ are the corresponding forms \eqref{Btheta1} and \eqref{Ftheta} with $\theta|_K=h_K^2$ for any $K\in\cT$. Based on the discussions in Section \ref{analysis}, we have the well-posedness of discrete problem \eqref{disweak2}. Let $\tri(\cdot,\cdot)\tri_2$ be the norm defined in \eqref{thetanorm} with $\theta|_K=h^2_K$, for $K\in\cT$. We also have the following locally robust and optimal a priori error estimate.

\begin{thm}\label{a_priori_mixed2}
Let $(\bsigma,u)$ be the solution of \eqref{eq_Darcy} and $(\bsigma_{2,h},u_{2,h})$ be the solution of problem \eqref{disweak2}, respectively. The following best approximation result is true:
\begin{eqnarray}\label{apriori2}
\tri(\bsigma-\bsigma_{2,h},u-u_{2,h})\tri_2 &\leq& 2\inf_{(\btau_h,v_h)\in \Sigma_{h,N}\times V_{h,D}}\tri(\bsigma-\btau_h,u-v_h)\tri_2.
\end{eqnarray}
For the $RT_{0,N}\times S_{1,D}$ approximation, under Assumption \ref{asmp_A} on the coefficients, if we further assume that $u|_K\in H^{1+s_K}(K)$, $(\nabla u-\bff)|_K\in H^{q_K}(K)^d$, $g|_K\in H^{t_K}(K)$ for $K\in\cT$, where the local regularity indexes $s_K$, $q_K$, and $t_K$ satisfy the following assumptions: $0<s_K\leq 1$ in two dimensions and $1/2<s_K\leq 1$ in three dimensions,  $1/2<q_K\leq 1$ with the constant $C_{rt}>0$ being unbounded as $q_K\downarrow 1/2$, and $0<t_K\leq 1$, then the following local robust and local optimal a priori error estimate holds: there exists a constant $C$ independent of $\a$ and the mesh-size, such that
\beq \label{a_priori_mixed_2}
\tri(\bsigma-\bsigma_{2,h},u-u_{2,h})\tri_2 \leq C\sum_{K\in\cT}\a_K^{1/2}\left(h_K^{s_K}|\nabla u|_{s_K,K}+ C_{rt} h_K^{q_K}|\nabla u-\bff|_{q_K,K} + \a_K^{-1}h_K^{1+t_K}|g|_{t_K,K}
\right).
\eeq
For the $BDM_{1,N}\times S_{2,D}$ approximation, under Assumption \ref{asmp_A} on the coefficients, if we further assume that $u\in H^{3}(\O)$, $(\nabla u-\bff)|_K\in H^{2}(K)^d$, and $g|_K\in H^{1}(K)$ for $K\in\cT$, then the following local robust and local optimal a priori error estimate holds: there exists a constant $C$ independent of $\a$ and the mesh-size, such that
\beq \label{a_priori_mixed_2bdm}
\tri(\bsigma-\bsigma_{2,h},u-u_{2,h})\tri_2 \leq C\sum_{K\in\cT}\a_K^{1/2}h_K^2\left(|\nabla u|_{2,K}+|\nabla u-\bff|_{2,K} + \a_K^{-1}|g|_{1,K}
\right).
\eeq
\end{thm}
The proof of the theorem is almost identical to that of Theorem \ref{a_priori_mixed1} with some necessary changes due to the extra factor $h$. 

\begin{rem}
We give some explanations that why the mesh-weighted second augmented mixed formulation is suggested. In the standard mixed formulation \eqref{eq_dm} for Darcy's equation, the divergence of the flux $\gradt\bsigma = g$ is a known quantity. For the standard dual mixed formulation \eqref{eq_dm}, we can easily derive the robust best approximation in the $L^2$-norm of the flux alone without invoking the approximation of the divergence of $\bsigma$ and the smoothness of $g$, see for example, Theorems 2 and 3 of \cite{Zhang:20mixed}. On the contrary, for the first augmented mixed formulation \eqref{disweak1}, its a priori error estimate \eqref{a_priori_mixed_1} is done for the combined norm $\tri\cdot\tri_1$. Thus, the error of the flux measured in the $L^2$-norm is influenced by the approximation of the divergence of $\bsigma$, which depends on the regularity of $g$ on the element $K\in\cT$. For example, for the case $\bff=0$ and  $u\in H^2(\O)$, we get $s_K=q_K=1$ for all $K\in\cT$ in \eqref{a_priori_mixed_1}. However, if the regularity of $g$ is low ($<1$), then the error of the flux measured in the $L^2$-norm is dominated by the bad approximation of the divergence of $\bsigma$, which is worse than the standard mixed formulation. Such sub-optimal result also appears in the LSFEM \eqref{lsfem1h} since all the terms are also coupled in \eqref{lsfem1h}. 

On the other hand, for the second augmented mixed method, the convergence order of $\|\a^{1/2}\nabla(u-u_{2,h})\|_0$ and $\|\a^{-1/2}\bsigma-\bsigma_{2,h})\|_0$ can  still be optimal, even if the regularity of $g$ is low. For example, for the case $\bff=0$ and  $u\in H^2(\O)$, as long as the regularity of $g|_K$ on each element $K\in\cT$, $t_K>0$, we can still get order $h$ convergence in \eqref{a_priori_mixed_2} for the $RT_{0,N}\times S_{1,D}$ approximation. For the $BDM_{1,N}\times S_{2,D}$ approximation, assuming that $u\in H^{3}(\O)$ and $(\nabla u-\bff)|_K\in H^{3}(K)^d$, we only requires $g|_K\in H^{1}(K)$ to get the optimal convergence.
\end{rem}

\subsection{A posteriori error analysis}
Define the global a posteriori error estimator:
\beq\label{ap_aug2}
\eta_2(\bsigma_{2,h},u_{2,h}) = \Big\{\|\a^{-1/2}h(g-\gradt\bsigma_{2,h})\|^2_{0} + \|\a^{1/2}(\bff-\nabla u_{2,h}-\a^{-1}\bsigma_{2,h})\|^2_{0}\Big\}^{1/2},
\eeq
and its local indicator 
\begin{eqnarray*}
\eta_{2,K}(\bsigma_{2,h},u_{2,h}) &=& \Big\{\|\a^{-1/2}h_K(g-\gradt\bsigma_{2,h})\|^2_{0,K} + \|\a^{1/2}(\bff-\nabla u_{2,h}-\a^{-1}\bsigma_{2,h})\|^2_{0,K}\Big\}^{1/2}.
\end{eqnarray*}

\begin{thm}\label{posteriori2}
Let $(\bsigma,u)$ be the solution of \eqref{eq_Darcy} and $(\bsigma_{2,h},u_{2,h})$ be the solution of problem \eqref{disweak2}, respectively. Assume that Assumption \ref{asmp_A} on the coefficients and Assumption \ref{asmp_QMA} (QMA) are true, then there exists positive constants $C$ independent of $\a$ and the mesh-size such that the following reliability holds:
\beq \label{rel_eta2}
\tri(\bsigma-\bsigma_{2,h}, u-u_{2,h})\tri_2 \leq C\eta_2(\bsigma_{2,h},u_{2,h}).
\eeq
\end{thm}
\begin{proof}
Let $\bE_2 = \bsigma-\bsigma_{2,h}$ and $e_2= u-u_{2,h}$. In \eqref{error_rep}, let $\theta|_K=h^2_K$ and $v_h = I_{rcl}e_2$, we have
$$
\tri(\bE_2,e_2)\tri^2_2 =  (\bff-\a^{-1}\bsigma_{2,h}-\nabla u_{2,h},\bE_2+\a\nabla (e_2- I_{rcl} e_2))+(g-\gradt\bsigma_{2,h},\a^{-1}h^2\gradt\bE_2+2(e_2- I_{rcl} e_2)).
 $$
Applying Cauchy-Schwarz and triangle inequalities, we get
\begin{eqnarray} \label{EE2}
\tri(\bE_2,e_2)\tri^2_2 &\leq& \|\a^{1/2}(\bff-\a^{-1}\bsigma_{2,h}-\nabla u_{2,h})\|_0(\|\a^{-1/2}\bE_2\|_{0}+\|\a^{1/2}\nabla (e_2- I_{rcl} e_2)\|_0)\\ \nonumber
&&+\|h\a^{-1/2}(g-\gradt\bsigma_{2,h})\|_{0} (\|h\a^{-1/2}\gradt\bE_2\|_0 +2\|h^{-1}\a^{1/2}(e_2- I_{rcl} e_2)\|_0.
\end{eqnarray}
By the robust Cl\'{e}ment interpolation result \eqref{rcl} under Assumption \ref{asmp_QMA}. we have the following robust results,
$$
\|\a^{1/2}\nabla (e_2- I_{rcl} e_2)\|_0 \leq C \|\a^{1/2}\nabla e_2\|_0 \quad\mbox{and}\quad
\|\a^{1/2}(e_2- I_{rcl} e_2)\|_0 \leq C \|h\a^{1/2}\nabla e_2\|_0.
$$ 
Substitute these two robust results into \eqref{EE2}, we get
\begin{eqnarray*} 
\tri(\bE_2,e_2)\tri^2_2 &\leq& C (\|\a^{1/2}(\bff-\a^{-1}\bsigma_{2,h}-\nabla u_{2,h})\|_0+\|h\a^{-1/2}(g-\gradt\bsigma_{2,h})\|_{0}) \\
&&(\|\a^{-1/2}\bE_2\|_{0}+\|h\a^{-1/2}\gradt\bE_2\|_{0}+\|\a^{1/2}\nabla e_2\|_0)\\ 
&\leq&C\eta_2(\bsigma_{2,h}, u_{2,h})\tri(\bE_2,e_2)\tri_2.
\end{eqnarray*}
The robust reliability result \eqref{rel_eta2} is proved. 
\end{proof}

Then using similar techniques of the proof of Theorem \ref{efficiency}, we give the efficiency for the error estimators or indicators of $\eta_{2,K}(\bsigma_{2,h},u_{2,h})$.
\begin{thm}\label{efficiency2}
Let $(\bsigma,u)$ be the solution of \eqref{eq_Darcy} and $(\bsigma_{2,h},u_{2,h})$ be the solution of problem \eqref{disweak2}, respectively.. Assume that Assumption \ref{asmp_A} on the coefficients is true, we have the following efficiency:
\beq\label{eff_eta2u}
\eta_{2,K}(\bsigma_{2,h},u_{2,h}) \leq \sqrt{2} \tri(\bsigma-\bsigma_{2,h},u-u_{2,h})\tri_{2,K},
\quad \forall K\in\cT.
\eeq
\end{thm}
%\begin{proof}
%By the triangle inequality and the first-order system \eqref{eq_Darcy}, for any $K\in\cT$, we have
%\begin{eqnarray*}
%\eta_{2,K}^2(\bsigma_{2,h},u_{2,h}) &=& \|h\a^{-1/2}(g-\gradt\bsigma_{2,h})\|^2_{0,K} 
%	+ \|\a^{1/2}\bff-\a^{-1/2}\bsigma_{2,h}-\a^{1/2}\nabla u_{2,h}\|^2_{0,K}\\
%&=& \|h\a^{-1/2}(\gradt\bsigma-\gradt\bsigma_{2,h})\|^2_{0,K} + \|\a^{-1/2}(\bsigma-\bsigma_{2,h})+\a^{1/2}\nabla (u- u_{2,h})\|^2_{0,K} \\
%&\leq& 
%\|\a^{-1/2}\gradt(\bsigma-\bsigma_{2,h})\|^2_{0,K} + 2\bigg( \|\a^{-1/2}(\bsigma-\bsigma_{2,h})\|^2_{0,K}+ \|\a^{1/2}\nabla (u- u_{2,h})\|^2_{0,K}  \big)\\
%&=& 2 \tri(\bsigma-\bsigma_{2,h},u-u_{2,h})\tri^2_{1,K}.
%\end{eqnarray*}
%The theorem is proved.
%\end{proof}
\subsection{Comparison with the mesh-weighted LSFEM} \label{comparison_2}
We have a corresponding mesh-weighted least-squares method for the second augmented mixed formulation.
Define the mesh-weighted least-squares functional as
\beq \label{lsfunctional_h}
J_h(\btau,v;\bff,g) := \|\a^{1/2}\nabla v +\a^{-1/2}\btau - \a^{1/2}\bff\|_0^2+  \|h\a^{-1/2}(\gradt\btau - g)\|_0^2, \quad (\btau,v)\in \bX.
\eeq
Then the mesh-weighted $L^2$-based least-squares minimization problem is: find $(\bsigma,u) \in \bX$, such that 
\beq  %\label{lsfem2}
J_h(\bsigma,u;\bff,g) = \inf_{(\btau,v)\in \bX} J_h(\btau,v;\bff,g).
\eeq 
Equivalently, it can be written in a weak form as: find $(\bsigma,u) \in \bX$, such that 
\beq %\label{lsfem2h}
b_h((\bsigma,u), (\btau,v)) = (\bff,\btau+\a\nabla v)+ (h^2\a^{-1}g,\gradt\btau), \quad \forall (\btau,v) \in \bX,
\eeq
where the mesh-weighted least-squares bilinear form $b_h$ is defined as follows:
\beq
b_h((\bchi,w), (\btau,v)) = (\bchi+\nabla w,\btau+\a\nabla v)+ (h^2\a^{-1}\gradt\bchi,\gradt\btau), \quad \forall (\bchi,w), (\btau,v) \in \bX.
\eeq
This kind of least-squares method is called the weighted $L^2$-discrete-least-squares principle; see Section 5.6.1 of \cite{BG:09}.

Using the same approximation as the second augmented mixed formulation, the mesh-weighted least-squares finite element method is: find $(\bsigma_{h}^{hls},u_h^{hls}) \in \Sigma_{h,N}\times V_{h,D}$, such that
\beq \label{lsfem2}
J_h(\bsigma^{hls}_h,u^{hls}_h;\bff,g) = \inf_{(\btau,v)\in \Sigma_{h,N}\times V_{h,D}} J_h(\btau,v;\bff,g).
\eeq 
Or, equivalently, find $(\bsigma_{h}^{hls},u_h^{hls}) \in \Sigma_{h,N}\times V_{h,D}$, such that,
\beq  \label{lsfem2h}
b_2(\bsigma_{h}^{hls},u_h^{hls}, (\btau,v)) = (\a^{-1}\bff,\btau+\a\nabla v)+ (h^2\a^{-1}g,\gradt\btau), \quad \forall (\btau,v) \in \Sigma_{h,N}\times V_{h,D}.
\eeq
The mathematical theory of the mesh-weighted LSFEM \eqref{lsfem2} or \eqref{lsfem2h} is much less satisfactory. Most importantly, we only have the following quasi-norm equivalence with $C_1$ and $C_2$ independent of the mesh size:
\beq \label{ls_equvalience2}
C_{1} h_{min}^2 \tri(\btau,v)\tri_2 ^2 \leq J_h(\btau,v;0,0) \leq C_{2} \tri(\btau,v)\tri_2 ^2, \quad \forall (\btau,v)\in \bX,
\eeq
where $h_{min} = \min \{ h_K, K\in\cT \}$. The coercivity can be easily derived from \eqref{ls_equvalience} and the definition of the norm. With only \eqref{ls_equvalience2} available, we can not expect mesh-independent a priori and a posteriori error estimates in the standard norm $\tri \cdot\tri_2$.

Another way to establish the analysis for the mesh-weighted LSFEM \eqref{lsfem2h} is to adopt the non-standard least-squares norm. Define 
\beq
\tri (\btau,v) \tri_{hls} = J_h^{1/2}(\btau,v;0,0) = ( \|\a^{1/2}\nabla v +\a^{-1/2}\btau\|_0^2+  \|h\a^{-1/2}\gradt\btau\|_0^2)^{1/2}, \quad (\btau,v)\in \bX.
\eeq
It is easy to see that 
$$
\tri (\btau,v) \tri_{hls} \leq \sqrt{2} \tri (\btau,v) \tri_{2}.
$$
Then we can establish the following best approximation a priori error estimate for the least-squares induced norm $\tri (\btau,v) \tri_{hls}$:
\beq\label{lowerbdhls}
\tri (\bsigma-\bsigma_{h}^{hls},u-u_{h}^{hls}) \tri_{hls} = \inf_{(\btau_h,v_h)\in \Sigma_{h,N}\times V_{h,D}} \tri (\bsigma-\btau_h,u-v_h) \tri_{hls} \leq  \sqrt{2} \inf_{(\btau_h,v_h)\in \Sigma_{h,N}\times V_{h,D}} \tri (\bsigma-\btau_h,u-v_h) \tri_2.
\eeq
Then we can get similar convergence results as in Theorem \ref{a_priori_mixed2}.

Let $(\bsigma_{a},u_{a}) \in \bX$, we can define the following mesh-weighted least-squares a posteriori error estimator:
\beq\label{ap_lsfemh}
\eta_{hls}(\bsigma_{a},u_{a}) := \Big (\|h\a^{-1/2}(g-\gradt\bsigma_{a})\|^2_{0} + \|\a^{1/2}(\bff-\nabla u_{a}-\a^{-1}\bsigma_{a})\|^2_{0}\Big )^{1/2} = J_h(\bsigma_{a},u_{a};\bff,g)^{1/2}.
\eeq
Let $\bE_a = \bsigma - \bsigma_{a}$ and $e_a = u-u_a$. We have the following identity:
\begin{eqnarray}\label{etahlsJh}
\eta_{hls}^2(\bsigma_{a},u_{a}) = J_h(\bsigma_a,u_a;\bff,g)  = J_h(\bE_a,e_a;0,0) = \tri (\bsigma-\bsigma_a,u-u_a) \tri_{hls}^2. 
\end{eqnarray}
Thus, if we are satisfied with the non-traditional mesh-weighted least-squares norm $\tri \cdot \tri_{hls}$, then we can use $\eta_{hls}^2(\bsigma_{h}^{hls},u_{h}^{hls})$ as the a posteriori error estimator for the method \eqref{lsfem2h}.  

In summary, we have robust and local optimal a priori error estimates for the mesh-weighted augmented mixed formulation with less regularity requirement on $g$. The mesh-weighted least-squares a posteriori error estimator is a robust error estimator for the method. This is an improvement for the mesh-weighted least-squares formulation where neither the a priori nor a posteriori error estimates are mesh-size and coefficient robust with respect to the standard norms.

\begin{rem}
The mesh-weighted least-squares formulation can be viewed as a practical version of the $H^{-1}$-least-squares method, see \cite{BLP:97}. 
Define the $H^{-1}$ least-squares functional as
\beq \label{lsfunctional_-1}
J_{-1}(\btau,v;\bff,g) := \|\a^{1/2}\nabla v +\a^{-1/2}\btau - \a^{1/2}\bff\|_0^2+  \|\a^{-1/2}(\gradt\btau - g)\|_{-1}^2, \quad (\btau,v)\in \bX.
\eeq
%We then have the following norm-equivalence: 
%\beq C_1(\|\a^{-1/2}\btau\|_0^2 + \|\a^{1/2}\nabla v\|_0^2)^{1/2} \leq 
%J_{-1}(\btau,v;0,0) \leq 
%C_2(\|\a^{-1/2}\btau\|_0^2 + \|\a^{1/2}\nabla v\|_0^2)^{1/2}, \quad (\btau,v)\in \bX.
%\eeq
 A multilevel preconditioner of the Laplace operator is proposed to replace the $H^{-1}$ norm in \cite{BLP:97}.
%Due to the inverse inequality, we have $h_{min}\|\gradt \btau_h\|_0 \leq C \|\gradt \btau_h\|_{-1}$, for all $\btau_h \in \Sigma_{h,N}$.  
To avoid the multilevel computation, a very coarse replacement is to replace the $H^{-1}$-norm with the mesh-weighted norm. Of course, we can only get a mesh-dependent norm-equivalence \eqref{ls_equvalience2}.
\end{rem}

\section{Numerical Experiments}
\setcounter{equation}{0}
In this section, we present serval numerical experiments to verify our findings in previous sections. Our main test problem is the interface problem with discontinuous coefficients from \cite{Kellogg:74,CD:02,CZ:09}. Let $\O=(-1,1)^2$ and let $\widetilde{u}(r,\theta)=r^\g\mu(\theta)$
%\beq\label{exactubar}
%\widetilde{u}(r,\theta)=r^\g\mu(\theta)
%\eeq
in polar coordinates with
\begin{equation}\label{mutheta}
\mu(\theta)=\left\{
\begin{array}{ll}
\cos((\pi/2-\phi)\g)\cdot((\theta-\pi/2+\rho)\g) & \mbox{ if }0\leq\theta\leq \pi/2,
 \\
\cos(\rho \g)\cdot\cos((\theta-\pi+\phi)\g) & \mbox{ if }\pi/2\leq\theta\leq\pi,
\\
\cos(\rho \phi)\cdot\cos((\theta-\pi-\rho)\g) & \mbox{ if }\pi\leq\theta\leq3\pi/2,
\\
\cos((\pi/2-\rho)\g)\cdot((\theta-3\pi/2-\phi)\g) & \mbox{ if }3\pi/2\leq\theta\leq 2\pi.
\end{array}
\right.
\end{equation}
The coefficient  $\a$ is
\begin{equation}\notag
\a(x)=\left\{
\begin{array}{ll}
R & \mbox{ in } (0,1)^2\cup(-1,0)^2,
 \\
1 & \mbox{ in }\O\backslash([0,1]^2\cup[-1,0]^2).
\end{array}
\right.
\end{equation}
We choose the numbers $\g$, $\rho$, $\phi$, and $R$ such that $\gradt (\a\nabla \widetilde{u}) = 0$ in $\O$. The following nonlinear relations of  $\g$, $\rho$, $\phi$, and $R$ can be found in \cite{CD:02}:
\begin{equation}\label{nonlinear}
\left\{
\begin{array}{l}
R = -\tan((\pi/2-\phi)\g)\cdot\cot(\rho\g),\\
1/R = -\tan(\rho\g)\cdot\cot(\phi\g),\\
R = -\tan(\phi\g)\cdot\cot((\pi/2-\rho)\g),\\
0<\g<2,\\
\max(0,\pi\g-\pi)<2\g\rho<\min(\pi\g,\pi),\\
\max(0,\pi-\pi\g)<-2\g\phi<\min(\pi,2\pi-\pi\g).
\end{array}
\right.
\end{equation}
The function $\widetilde{u}(r,\theta)\in H^{1+\gamma-\epsilon}(\O),$ for any $\epsilon>0$. The regularity index $\g$ is less than $1$, The function $\widetilde{u}$  is singular at the origin. 
We give serval examples of the coefficients $\a$ and the corresponding numbers $\g$, $\rho$, and $\phi$ in Table \ref{tab_numbers}. These numbers can be computed by solving \eqref{nonlinear} using Newton's method. Some of them can also be found in \cite{CD:02}. 
\begin{table}[h]
\caption{Numbers of $\phi,\g$ and $R$ with $\rho = \pi/4.$}\label{tab_numbers}
\begin{tabular}{|c|c|c|c|}
\hline
       &$\phi$& $\g$  & $R$      \\\hline
$\mbox{Data}1$ & $-2.3561944901923448$& $0.50$  &  $5.82842712474619$ \\\hline
$\mbox{Data}2$ & $-7.06858347058882$& $0.20$  &  $39.8634581884533$   \\\hline
$\mbox{Data}3$ & $-9.68657734859297$& $0.15$  &  $71.3848801304590$   \\\hline
$\mbox{Data}4$ & $-14.92256510455152$& $0.10$  &  $161.447638797588$   \\\hline
\end{tabular}
\end{table}

Let $u := \widetilde{u}+u_0$ with $u_0(x,y) = \left\{
\begin{array}{ll}
x+1,&x\leq 0,\\
1&x>0.
\end{array}
\right.
$
%\begin{equation}\notag
%u_0(x,y) = \left\{
%\begin{array}{ll}
%x+1,&x\leq 0,\\
%1&x>0.
%\end{array}
%\right.
%\end{equation}
Then $u$ is the solution of the following generalized Darcy's problem
\begin{equation}\label{eq_darcynumer}
\left\{
\begin{array}{lllll}
\gradt \bsigma    & =& 0 & \mbox{in } \O,
 \\[1mm]
\a\nabla u+ \bsigma  & =& \a\nabla u_0 & \mbox{in } \O.
\end{array}
\right.
\end{equation}
This will be our main problem to do numerical tests.

We propose the following criteria to test the robustness of the methods with respect to $\a$. 
 Let $(\bsigma_h,u_h)$ be the numerical solutions computed by the numerical methods (augmented mixed methods and LSFEMs) discussed in the paper; we want to compare the ratio of the error in the energy norm and its corresponding a posteriori error estimator. 
For the first kind augmented mixed methods and the $L^2$-LSFEM, we measure the error in $\tri (\bsigma-\bsigma_h,u-u_h) \tri_1$ (the definition in \eqref{thetanorm} with $\theta=1$) and the a posteriori error estimator is the one based on the $L^2$-LS-functional $\eta_{ls}(\bsigma_h,u_h)=J^{1/2}(\bsigma_h,u_h;0,0)$.  For the second kind augmented mixed methods and mesh-weighted LSFEM, we measure the error in $\tri (\bsigma-\bsigma_h,u-u_h) \tri_2$ (the definition in \eqref{thetanorm} with $\theta=h_K^2$) and the a posteriori error estimator is the one based on the mesh-weighted LS-functional $\eta_{hls}(\bsigma_h,u_h)=J^{1/2}_h(\bsigma_h,u_h;0,0)$.  

If the a posteriori error estimator is robust, then for different $\alpha$, the ratio or the so-called  effectivity index 
\beq
\mbox{eff-index} = \dfrac{\tri (\bsigma-\bsigma_h,u-u_h) \tri}{\eta}
\eeq
should be a constant.

The standard adaptive finite element method is based on the following loop: SOLVE, ESTIMATE, MARK, and REFINE. We use the D\"{o}rfler's bulk marking strategy. The relative error is computed by
$\mbox{rel-err} = \tri(\bsigma-\bsigma_h,u-u_h)\tri/\tri(\bsigma,u)\tri$.

We do not seek numerical examples to check the robust a priori error estimates.

\subsection{Convergence tests for adaptive augmented mixed methods with pure Dirichlet boundary conditions}
In this subsection, we use Data4 in  Table \ref{tab_numbers}. The main purpose is to show the convergence history and the final adaptive mesh for the augmented mixed methods. We choose the pure Dirichlet boundary condition. The parameter in  D\"{o}rfler's bulk marking strategy is $0.3$ and the stopping criteria is $\mbox{rel-err}\leq 0.010$.

In Figure \ref{fig_eta1mes1RT0}, we present the numerical test of the first adaptive augmented mixed method \eqref{disweak1}, with $RT_{0,N}\times S_{1,D}$ being the finite element space. 
On the left of Figure  \ref{fig_eta1mes1RT0}, we show the reference line $\mbox{Dofs}^{-1/2}$, the decay of the error estimator, and the error measured in $\tri(\cdot,\cdot)\tri_1$. On the right of Figure  \ref{fig_eta1mes1RT0}, the final mesh generated by $\eta_1(\bsigma_{1,h},u_{1,h})$ is presented after $75$ loops of bisection. The final DOF is $4921$. The convergence and the final mesh are both optimal, and the final mesh is similar to the results presented in \cite{CD:02,CZ:09}.

% and he error by the estimator in the final loop is $0.1340.$
\begin{figure}[h]
  \centering
  \begin{minipage}[t]{0.45\linewidth}
  \includegraphics[scale=0.38]{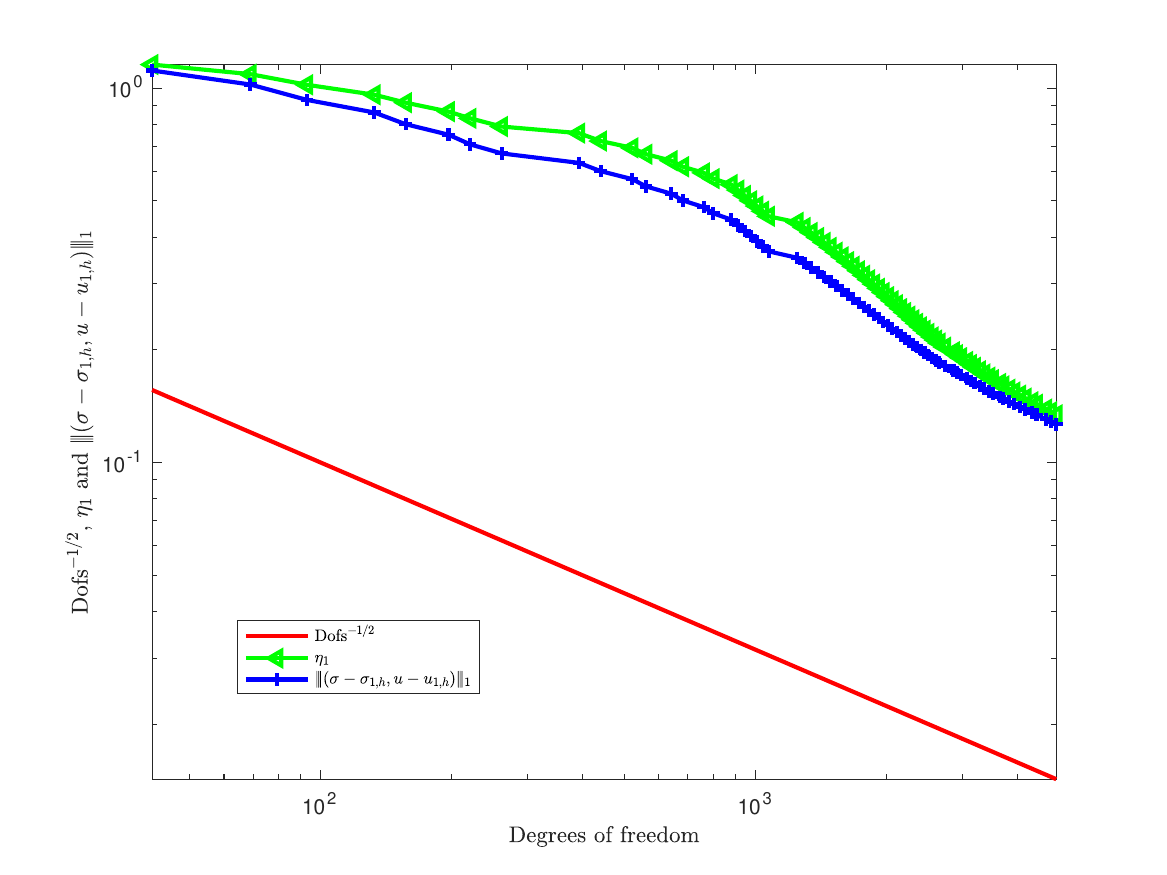}   %[width=7cm,height=5cm]
  \end{minipage}
  \begin{minipage}[t]{0.45\linewidth}
  \includegraphics[scale=0.5]{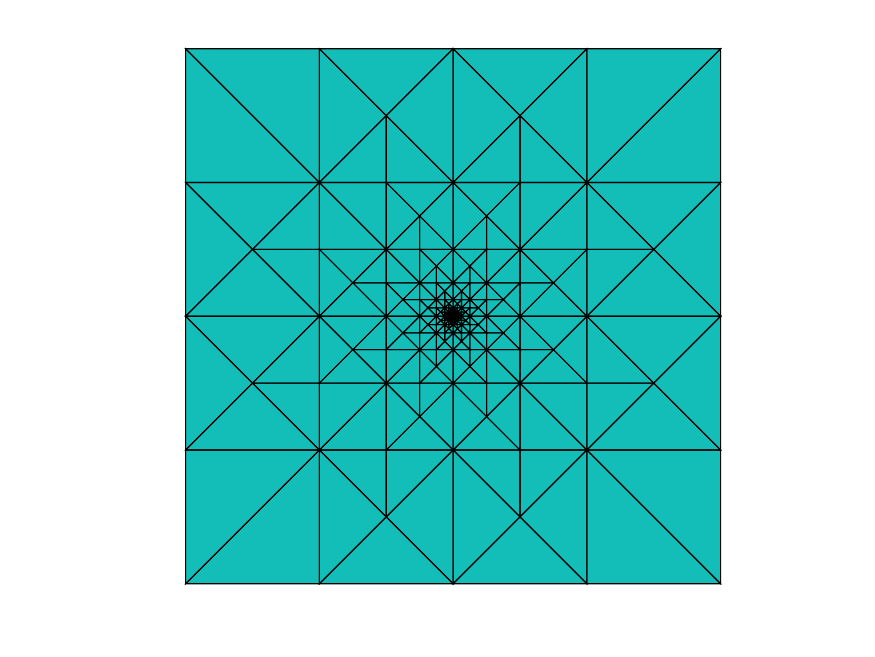}
  \end{minipage}
    \caption{Adaptive convergence results and final mesh for the first adaptive augmented mixed method \eqref{disweak1}}
    \label{fig_eta1mes1RT0}
\end{figure}

In Figure \ref{fig_eta2mesh2RT0}, we present the numerical test of the second (mesh-weighted) adaptive augmented mixed method \eqref{disweak2}, with $RT_{0,N}\times S_{1,D}$ being the finite element space. On the left of Figure  \ref{fig_eta2mesh2RT0}, we show the reference line $\mbox{Dofs}^{-1/2}$, the decay of the error estimator, and the error measured in $\tri(\cdot,\cdot)\tri_2$. On the right of Figure  \ref{fig_eta2mesh2RT0}, the final mesh generated by $\eta_2(\bsigma_{2,h},u_{2,h})$ is presented after $86$ loops of bisection. The final DOF is $4621$. The convergence and the final mesh are both optimal, and the final mesh is similar to the results presented in \cite{CD:02,CZ:09}.
\begin{figure}[h]
  \centering
  \begin{minipage}[t]{0.45\linewidth}
  \includegraphics[scale=0.38]{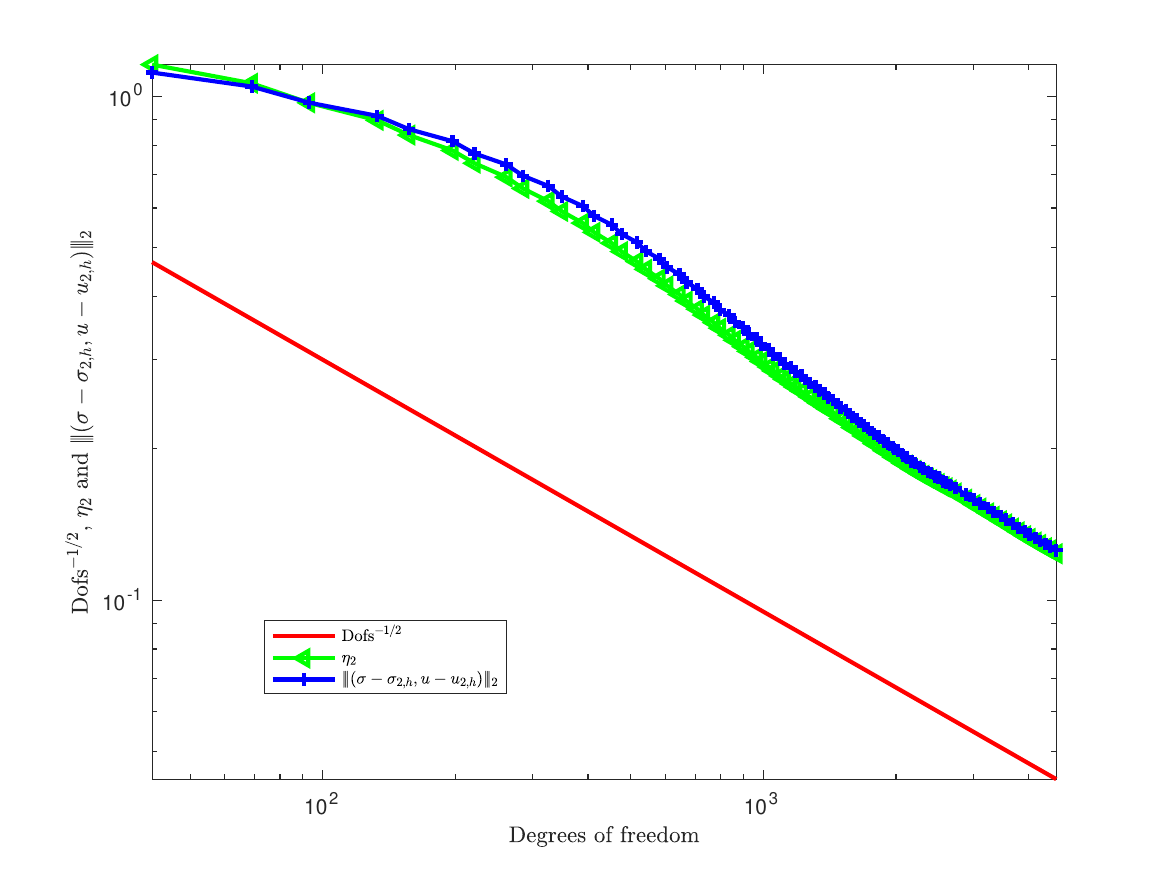}   %[width=7cm,height=5cm]
  \end{minipage}
  \begin{minipage}[t]{0.45\linewidth}
  \includegraphics[scale=0.5]{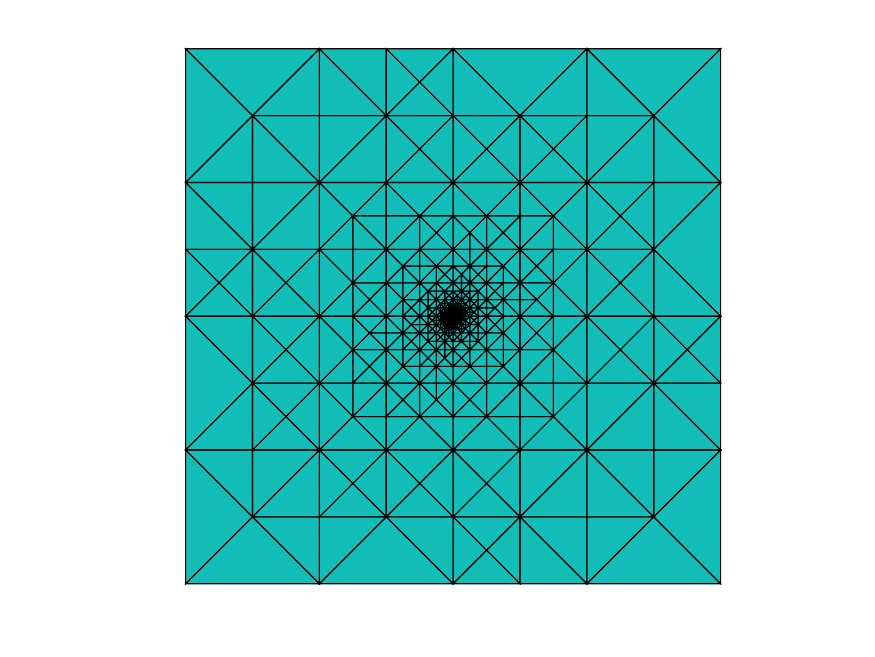}
  \end{minipage}
    \caption{Adaptive convergence results and final mesh for the second (mesh-weighted) adaptive augmented mixed method using $RT_{0,N}\times S_{1,D}$}
    \label{fig_eta2mesh2RT0}
\end{figure}

In Figure \ref{fig_eta2mesh2BDM1}, we present the numerical test of the second (mesh-weighted) adaptive augmented mixed method \eqref{disweak2}, with $BDM_{1,N}\times S_{2,D}$. The reference line in this case is $\mbox{Dofs}^{-1}$. The final DOFs is $1997$ after $49$ times of bisection. 
\begin{figure}[h]
  \centering
  \begin{minipage}[t]{0.45\linewidth}
  \includegraphics[scale=0.38]{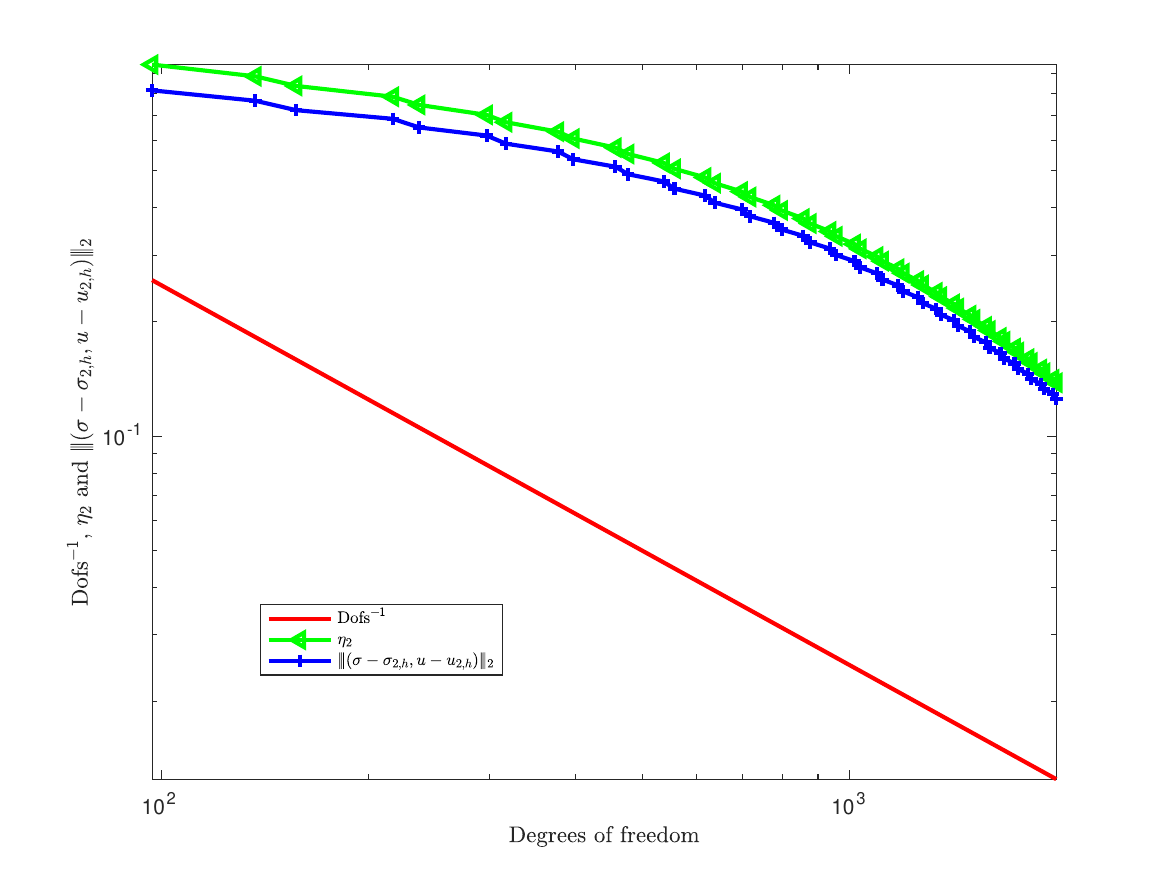}   %[width=7cm,height=5cm]
  \end{minipage}
  \begin{minipage}[t]{0.45\linewidth}
  \includegraphics[scale=0.5]{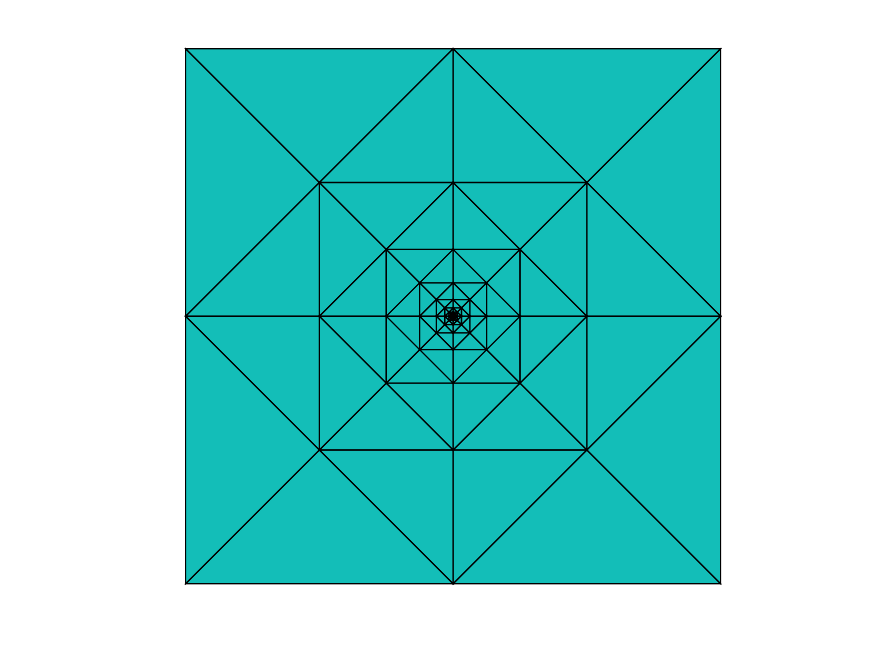}
  \end{minipage}
    \caption{Adaptive convergence results and final mesh for the second (mesh-weighted) adaptive augmented mixed method using $BDM_{1,N}\times S_{2,D}$}
    \label{fig_eta2mesh2BDM1}
\end{figure}

\subsection{Robustness tests for adaptive augmented mixed methods with pure Dirichlet boundary conditions}
In this subsection, we present the numerical results of the effectivity index for adaptive augmented mixed methods with pure Dirichlet boundary conditions for different $\a$ to check the robustness of the methods. The same marking strategy and stopping criteria of the previous subsection are used.
From Table \ref{tab_AugRT_PureD} to Table \ref{tab_AughBDM_PureD}, we show the eff-indexes for different $\a$ for the first and second augmented methods. As seen from these tables, the eff-index is almost a constant for different $\a$; this verifies least-squares a posteriori error estimators for both augmented mixed methods are robust. 
\begin{table}[htbp]
\caption{The first adaptive augmented mixed method for different $\a$, pure Dirichlet BC: eff-index, number of refinements $k$, number of elements $n$, the final $\eta_1(\bsigma_{1,h},u_{1,h})$,  and the final error $\tri(\bsigma-\bsigma_{1,h},u-u_{1,h})\tri_1$}
\label{tab_AugRT_PureD}
\begin{tabular}{|c|c|c|c|c|c|}
\hline
&$\mbox{eff-index}$& $k$  & $n$  & $\eta_1(\bsigma_{1,h},u_{1,h})$ &
$\tri(\bsigma-\bsigma_{1,h},u-u_{1,h})\tri_1$\\\hline
$\mbox{Data}1$ & $1.0006$& $34$  &  $15824$ & $0.0250$ & $0.0250$ \\\hline
$\mbox{Data}2$ & $1.0075$& $64$  &  $7216$ & $0.0621$ & $0.0616$ \\\hline
$\mbox{Data}3$ & $1.0179$& $71$  &  $4648$ & $0.0842$& $0.0828$ \\\hline
$\mbox{Data}4$ & $1.0605$& $75$  &  $2448$ & $0.1340$& $0.1264$ \\\hline
\end{tabular}
\end{table}

\begin{table}[htbp]
\caption{The second (mesh-weighted) adaptive augmented mixed method using $RT_{0,N}\times S_{1,D}$ for different $\a$, pure Dirichlet BC: eff-index, number of refinements $k$, number of elements $n$, the final $\eta_2(\bsigma_{2,h},u_{2,h})$,  and the final error $\tri(\bsigma-\bsigma_{2,h},u-u_{2,h})\tri_2$}
\label{tab_AughRT_PureD}
\begin{tabular}{|c|c|c|c|c|c|}
\hline
&$\mbox{eff-index}$& $k$  & $n$  & $\eta_2(\bsigma_{2,h},u_{2,h})$ &
$\tri(\bsigma-\bsigma_{2,h},u-u_{2,h})\tri_2$\\\hline
$\mbox{Data}1$ & $0.9963$& $35$  &  $15184$ & $0.0254$ & $0.0255$ \\\hline
$\mbox{Data}2$ & $0.9966$& $70$  &  $7088$ & $0.0616$ & $0.0618$ \\\hline
$\mbox{Data}3$ & $0.9949$& $80$  &  $4524$ & $0.0823$& $0.0827$ \\\hline
$\mbox{Data}4$ & $0.9847$& $86$  &  $2300$ & $0.1236$& $0.1256$ \\\hline
\end{tabular}
\end{table}

\begin{table}[htbp]
\caption{The second (mesh-weighted) adaptive augmented mixed method using $BDM_{1,N}\times S_{2,D}$ for different $\a$, pure Dirichlet BC: eff-index, number of refinements $k$, number of elements $n$, the final $\eta_2(\bsigma_{2,h},u_{2,h})$,  and the final error $\tri(\bsigma-\bsigma_{2,h},u-u_{2,h})\tri_2$}
\label{tab_AughBDM_PureD}
\begin{tabular}{|c|c|c|c|c|c|}
\hline
&$\mbox{eff-index}$& $k$  & $n$  & $\eta_2(\bsigma_{2,h},u_{2,h})$ &
$\tri(\bsigma-\bsigma_{2,h},u-u_{2,h})\tri_2$\\\hline
$\mbox{Data}1$ & $1.0728$& $20$  &  $320$ & $0.0262$ & $0.0244$ \\\hline
$\mbox{Data}2$ & $1.0792$& $43$  &  $416$ & $0.0659$ & $0.0611$ \\\hline
$\mbox{Data}3$ & $1.1068$& $47$  &  $380$ & $0.0929$& $0.0839$ \\\hline
$\mbox{Data}4$ & $1.1014$& $49$  &  $396$ & $0.1383$& $0.1256$ \\\hline
\end{tabular}
\end{table}

\subsection{Convergence and robustness tests for the adaptive $L^2$-LSFEM with pure Dirichlet boundary conditions}
In this subsection, we present the numerical results of convergence result of the adaptive $L^2$-LSFEM  \eqref{lsbform} with pure Dirichlet boundary conditions and the effectivity index for different $\a$. 

In Figure \ref{fig_line_LSRT_D4_PurD}, we present the numerical test of the adaptive $L^2$-LSFEM \eqref{lsbform} with pure Dirichlet boundary condition with Data4. The convergence and the final mesh are both optimal with enough mesh grids, and the final mesh is similar to the results presented in \cite{CD:02,CZ:09}.

\begin{figure}[htbp]
  \centering
  \begin{minipage}[t]{0.45\linewidth}
  \includegraphics[scale=0.38]{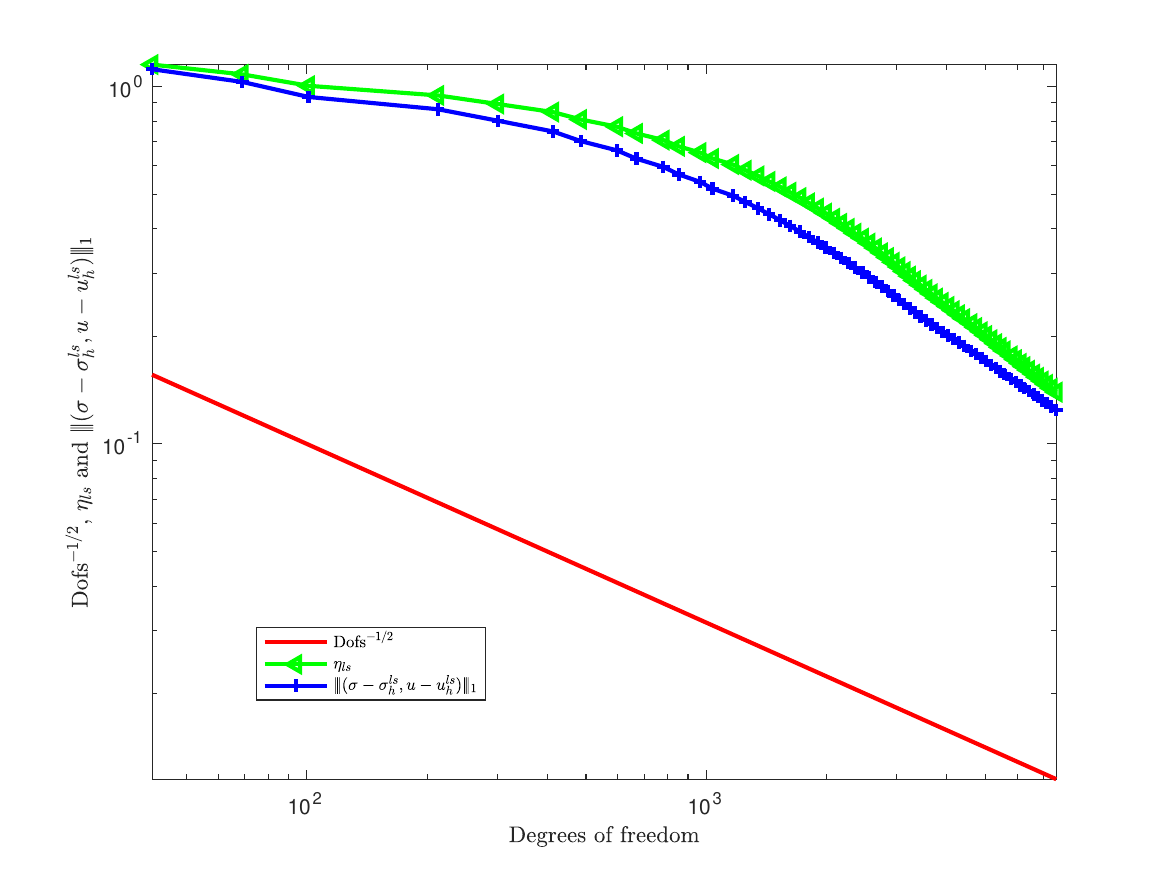}   %[width=7cm,height=5cm]
  \end{minipage}
  \begin{minipage}[t]{0.45\linewidth}
  \includegraphics[scale=0.5]{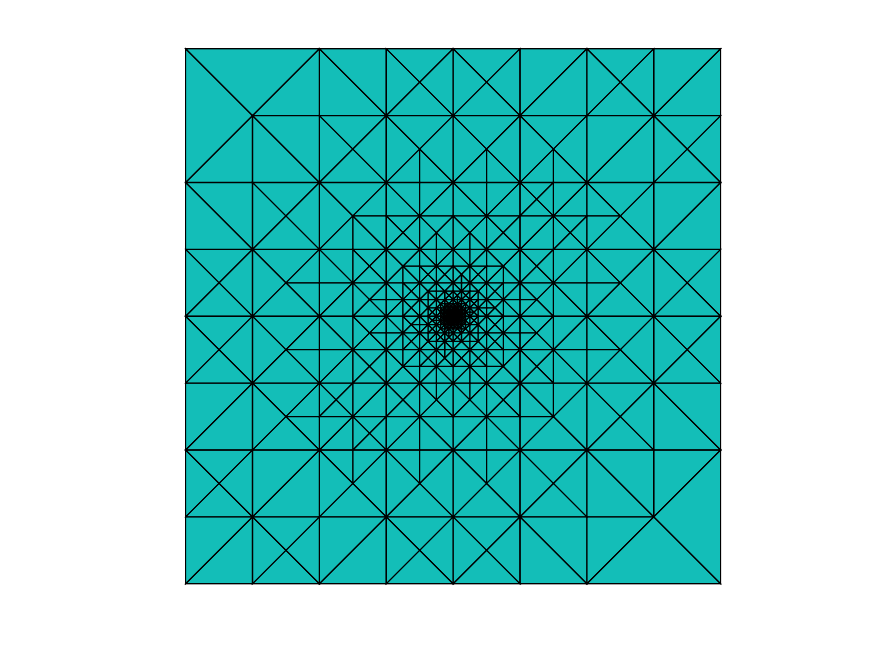}   %[width=7cm,height=5cm]
  \end{minipage}
    \caption{Adaptive convergence results and final mesh for the $L^2$-LSFEM with pure Dirichlet BC}
    \label{fig_line_LSRT_D4_PurD}
\end{figure}

We show the eff-indexes for different $\a$ for the $L^2$-LSFEM with pure Dirichlet BC in Table \ref{tab_LSRT_PureD}. As seen from these tables, the eff-index is almost a constant for different $\a$, this verifies that for this special case that each subdomain has a non-empty Dirichlet boundary condition, the $L^2$-LSFEM is robust. 

\begin{table}[htbp]
\caption{The $L^2$-LSFEM for different $\a$, pure Dirichlet BC: eff-index, number of refinements $k$, number of elements $n$, the final $\eta_{ls}(\bsigma^{ls}_{h},u^{ls}_{h})$,  and the final error $\tri(\bsigma-\bsigma^{ls}_h,u-u^{ls}_h)\tri_1$}
\label{tab_LSRT_PureD}
\begin{tabular}{|c|c|c|c|c|c|}
\hline
&$\mbox{eff-index}$& $k$  & $n$  & $\eta_{ls}(\bsigma^{ls}_h,u^{ls}_h)$ &
$\tri(\bsigma-\bsigma^{ls}_h,u-u^{ls}_h)\tri_1$\\\hline
$\mbox{Data}1$ & $1.0019$& $31$  &  $14476$ & $0.0263$ & $0.0262$ \\\hline
$\mbox{Data}2$ & $1.0171$& $58$  &  $7772$ & $0.0628$ & $0.0617$ \\\hline
$\mbox{Data}3$ & $1.0406$& $63$  &  $5400$ & $0.0871$& $0.0837$ \\\hline
$\mbox{Data}4$ & $1.1216$& $65$  &  $3744$ & $0.1395$& $0.1244$ \\\hline
\end{tabular}
\end{table}

\subsection{Robustness tests for adaptive augmented mixed methods with mixed boundary conditions}
In this subsection, we present the numerical results of the effectivity index of adaptive augmented mixed methods with mixed boundary conditions with different $\a$ to check the robustness of the methods.  

The mixed boundary are chosen as follows: $\Gamma_D= \{(x,y)\in\partial\O ; x\in(-1,1), y =-1\}$ and $\Gamma_N = \partial\O\setminus\Gamma_D$. The boundary conditions $g_D$ and $g_N$ are given by the true solution.

From Table \ref{tab_AugRT_PureD_mixbd} to Table \ref{tab_AughBDM_PureD_mixbd}, we show the eff-indexes for different $\a$ for the first and second augmented methods. As seen from these tables, the eff-index is almost a constant for different $\a$; this verifies least-squares a posteriori error estimators for augmented mixed methods are robust for the general mixed boundary condition case. 

\begin{table}[htbp]
\caption{The first adaptive augmented mixed method for different $\a$, mixed BC: eff-index, number of refinements $k$, number of elements $n$, the final $\eta_1(\bsigma_{1,h},u_{1,h})$,  and the final error $\tri(\bsigma-\bsigma_{1,h},u-u_{1,h})\tri_1$}
\label{tab_AugRT_PureD_mixbd}
\begin{tabular}{|c|c|c|c|c|c|}
\hline
&$\mbox{eff-index}$& $k$  & $n$  & $\eta_1(\bsigma_{1,h},u_{1,h})$ &
$\tri(\bsigma-\bsigma_{1,h},u-u_{1,h})\tri_1$\\\hline
$\mbox{Data}1$ & $1.0006$& $37$  &  $41031$ & $0.0156$ & $0.0155$ \\\hline
$\mbox{Data}2$ & $1.0058$& $73$  &  $19970$ & $0.0375$ & $0.0373$ \\\hline
$\mbox{Data}3$ & $1.0138$& $84$  &  $13622$ & $0.0497$& $0.0490$ \\\hline
$\mbox{Data}4$ & $1.0497$& $93$  &  $7605$ & $0.0795$& $0.0758$ \\\hline
\end{tabular}
\end{table}

\begin{table}[htbp]
\caption{The second (mesh-weighted) adaptive augmented mixed method using $RT_{0,N}\times S_{1,D}$ for different $\a$, mixed BC: eff-index, number of refinements $k$, number of elements $n$, the final $\eta_2(\bsigma_{2,h},u_{2,h})$,  and the final error $\tri(\bsigma-\bsigma_{2,h},u-u_{2,h})\tri_2$}
\label{tab_AughRT_PureD_mixbd}
\begin{tabular}{|c|c|c|c|c|c|}
\hline
&$\mbox{eff-index}$& $k$  & $n$  & $\eta_2(\bsigma_{2,h},u_{2,h})$ &
$\tri(\bsigma-\bsigma_{2,h},u-u_{2,h})\tri_2$\\\hline
$\mbox{Data}1$ & $0.9965$& $39$  &  $45355$ & $0.0147$ & $0.0147$ \\\hline
$\mbox{Data}2$ & $0.9966$& $79$  &  $19103$ & $0.0374$ & $0.0375$ \\\hline
$\mbox{Data}3$ & $0.9962$& $93$  &  $12478$ & $0.0492$& $0.0494$ \\\hline
$\mbox{Data}4$ & $0.9900$& $108$  &  $6046$ & $0.0748$& $0.0756$ \\\hline
\end{tabular}
\end{table}

\begin{table}[htbp]
\caption{The second (mesh-weighted) adaptive augmented mixed method using $BDM_{1,N}\times S_{2,D}$ for different $\a$, mixed BC: eff-index, number of refinements $k$, number of elements $n$, the final $\eta_2(\bsigma_{2,h},u_{2,h})$,  and the final error $\tri(\bsigma-\bsigma_{2,h},u-u_{2,h})\tri_2$}
\label{tab_AughBDM_PureD_mixbd}
\begin{tabular}{|c|c|c|c|c|c|}
\hline
&$\mbox{eff-index}$& $k$  & $n$  & $\eta_2(\bsigma_{2,h},u_{2,h})$ &
$\tri(\bsigma-\bsigma_{2,h},u-u_{2,h})\tri_2$\\\hline
$\mbox{Data}1$ & $1.0735$& $23$  &  $560$ & $0.0154$ & $0.0144$ \\\hline
$\mbox{Data}2$ & $1.0775$& $51$  &  $704$ & $0.0386$ & $0.0358$ \\\hline
$\mbox{Data}3$ & $1.0831$& $59$  &  $670$ & $0.0529$& $0.0489$ \\\hline
$\mbox{Data}4$ & $1.0553$& $70$  &  $573$ & $0.0800$& $0.0758$ \\\hline
\end{tabular}
\end{table}

\subsection{Non-robustness for $L^2$-LSFEM with mixed boundary conditions}
We present the non-robustness of the $L^2$-LSFEM \eqref{lsbform} with mixed boundary conditions. We use the same mixed boundary setting as the previous subsection.

As we can see from Table \ref{tab_LSRT_MixBd}, the eff-index is not a constant. This verifies the conclusion that the $L^2$-LSFEM is not robust with respect to $\a$. On the other hand, as we can see, the eff-index does not change as strongly as $\a$; this is explained in Remark \ref{lsfem_mild}.

\begin{table}[htbp]
\caption{The $L^2$-LSFEM for different $\a$, mixed BC: eff-index, number of refinements $k$, number of elements $n$, the final $\eta_{ls}(\bsigma^{ls}_{h},u^{ls}_{h})$,  and the final error $\tri(\bsigma-\bsigma^{ls}_h,u-u^{ls}_h)\tri_1$}
\label{tab_LSRT_MixBd}
\begin{tabular}{|c|c|c|c|c|c|}
\hline
&$\mbox{eff-index}$& $k$  & $n$  & $\eta_{ls}(\bsigma^{ls}_h,u^{ls}_h)$ &
$\tri(\bsigma-\bsigma^{ls}_h,u-u^{ls}_h)\tri_1$\\\hline
$\mbox{Data}1$ & $0.9972$& $80$  &  $14434$ & $0.0259$ & $0.0260$ \\\hline
$\mbox{Data}2$ & $0.8641$& $91$  &  $9542$ & $0.0528$ & $0.0611$ \\\hline
$\mbox{Data}3$ & $0.7079$& $110$  &  $8713$ & $0.0590$& $0.0833$ \\\hline
$\mbox{Data}4$ & $0.4787$& $139$  &  $11754$ & $0.0596$& $0.1244$ \\\hline
\end{tabular}
\end{table}

In Figure \ref{fig_line_LSRT_D3_mixbd}, we can also see that even for a fixed $\a$, the eff-index is not a constant when mesh is refined, which means it is not even robust with respect to the mesh-size. In contract, in Figure \ref{fig_line_LSRT_D4_PurD},  the eff-index is almost a constant for a fixed $\a$ for the special case of $L^2$-LSFEM where the robust Poincar\'e inequality holds.

\begin{figure}[htbp]
  \centering
  \begin{minipage}[t]{0.45\linewidth}
  \includegraphics[scale=0.4]{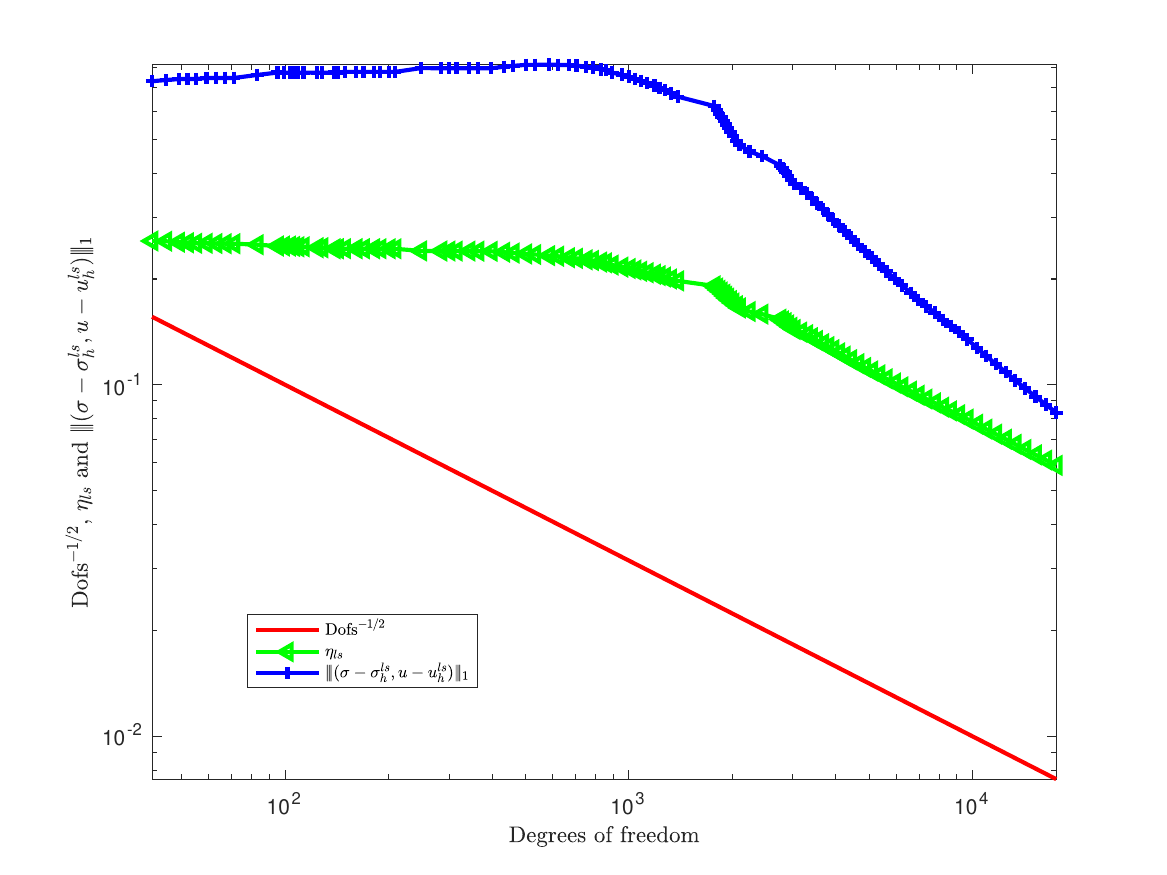}   %[width=7cm,height=5cm]
  \end{minipage}
    \caption{Adaptive convergence results for the $L^2$-LSFEM with mixed boundary conditions with Data3.}
    \label{fig_line_LSRT_D3_mixbd}
\end{figure}

\subsection{Numerical tests for mesh-weighted LSFEMs}
For the mesh-weighted LSFEM  \eqref{lsfem2h}, there are no rigorous a priori and a posteriori error estimates with respect to standard norms. We show the convergence result of adaptive convergence history for the problem with pure Dirichlet boundary conditions with Data4 using  $RT_{0,N}\times S_{1,D}$ with $39$ times of refinement in Figure \ref{fig_line_LShRT_PureD1}. The error measured in norm $\tri(\cdot,\cdot)\tri_2$ (the blue line in Figure \ref{fig_line_LShRT_PureD1}) does not decay as the a posteriori error estimator do. Also, the final mesh it obtained is skewed, which is a non-optimal case as discussed in \cite{CD:02,CZ:09}. This suggests that this method may not be a good choice for this class of problem.

\begin{figure}[htbp]
  \centering
  \begin{minipage}[t]{0.45\linewidth}
  \includegraphics[scale=0.38]{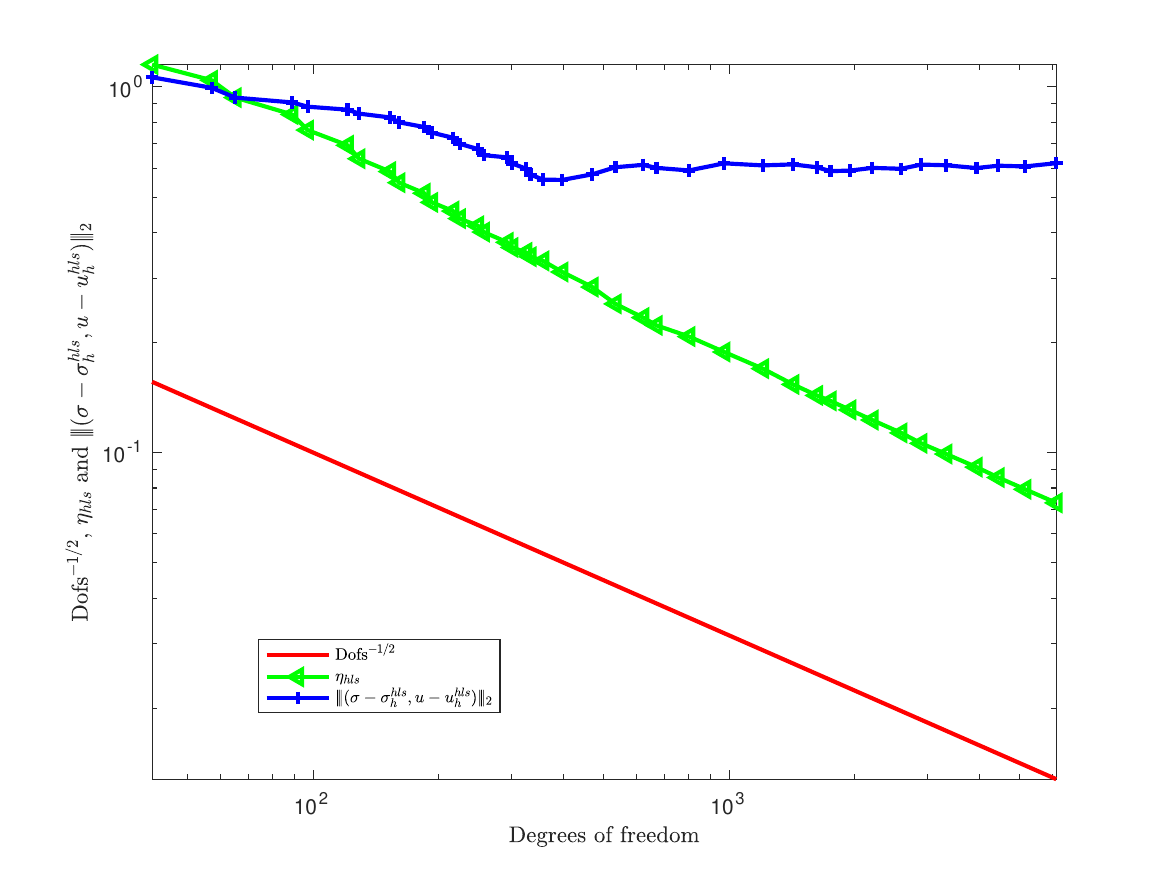}   %[width=7cm,height=5cm]
  \end{minipage}
  \begin{minipage}[t]{0.45\linewidth}
  \includegraphics[scale=0.5]{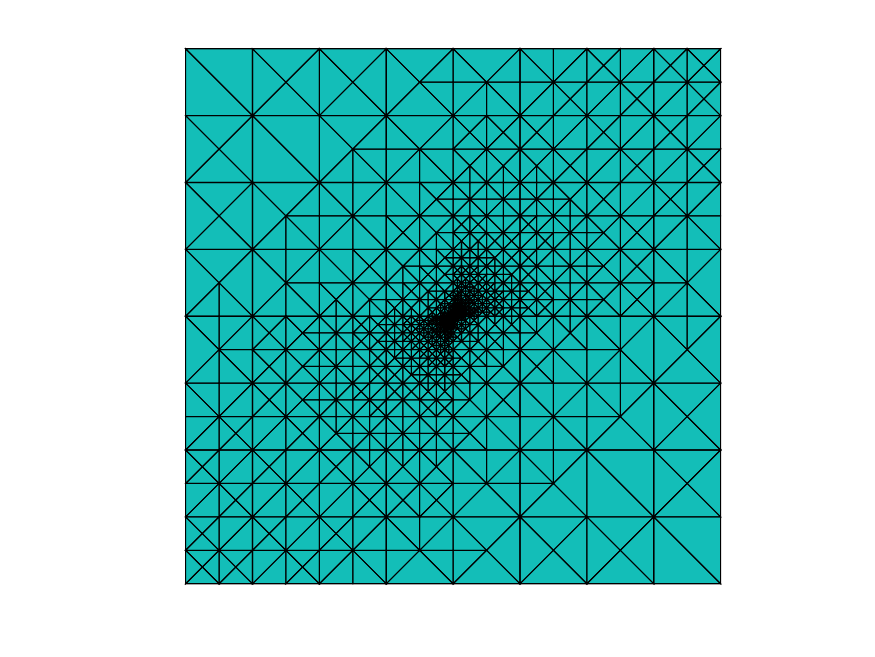}   %[width=7cm,height=5cm]
  \end{minipage}
     \caption{Adaptive convergence results and final mesh for the mesh-weighted LSFEM with pure Dirichlet boundary conditions.}
    \label{fig_line_LShRT_PureD1}
\end{figure}

\section{Final Comments}
In this paper, for the generalized Darcy problem, we study a special Galerkin-Least-Squares method, the augmented mixed finite element method, and its relationship to the standard least-squares finite element method (LSFEM). One of the paper's main contributions is to connect the augmented mixed finite element method and the bonafide LSFEM and discuss their shared properties and differences. Both methods share some good properties: both methods are based on the physically meaningful first-order system. Thus, important physical qualities can be approximated in their intrinsic spaces; both methods are coercive and stable, and their finite element discrete problems are coercive and stable as long as the discrete spaces are subspaces of the abstract spaces of the variational problems. Thus, no inf-sup condition of the discrete spaces and mesh size restriction are needed; both methods can use the least-squares functional as the build-in a posteriori error estimator.
On the other hand, the augmented mixed methods and the LSFEMs have their advantages. For the augmented mixed finite element methods, we show that the a priori and a posteriori error estimates are robust with respect to the coefficients of the problem. In contrast, we discuss the non-robustness of standard least-squares finite element methods. With the flexibility of being a partial least-squares method, it is possible that the augmented mixed finite element method can have better numerical properties than the bona-fide least-squares method. However, being a partial least-squares method, the augmented mixed method also loses one main property of the bonafide least-squares method, which may be very useful for minimization-based methods: the minimization of the least-squares energy, even though we can always associate a Ritz-minimization variational principle to the symmetric version of the augmented mixed method.

\bibliographystyle{plain}

\end{document}

%% file: defs.tex
\newcommand{\BX}{{\bf X}}
\newcommand{\cv}{{\cal V}}
\newcommand{\cW}{{\cal W}}
\newcommand{\co}{{\cal O}}

\renewcommand{\theequation}{\thesection.\arabic{equation}}
\def\@eqnnum{{\reset@font\rm (\theequation)}}

\def\abstract{
\advance \rightskip by 10mm
\advance \leftskip by 10mm
\vspace{-0.8em}
\noindent
\small{\bf Abstract.}
}
\def\endabstract{\par\normalsize\rm}

\def\Xint#1{\mathchoice
{\XXint\displaystyle\textstyle{#1}}%
{\XXint\textstyle\scriptstyle{#1}}%
{\XXint\scriptstyle\scriptscriptstyle{#1}}%
{\XXint\scriptscriptstyle\scriptscriptstyle{#1}}%
\!\int}
\def\XXint#1#2#3{{\setbox0=\hbox{$#1{#2#3}{\int}$}
\vcenter{\hbox{$#2#3$}}\kern-.5\wd0}}
\def\ddashint{\Xint=}
\def\dashint{\Xint-}

%Greek Letters
\def\a{\alpha}
\def\b{\beta}
\def\d{\delta}\def\D{\Delta}
\def\e{\epsilon}
\def\g{\gamma}\def\G{\Gamma}
\def\k{\kappa}
\def\lam{\lambda}\def\Lam{\Lambda}
\renewcommand\o{\omega}\renewcommand\O{\Omega}
\def\s{\sigma}\def\S{\Sigma}
\renewcommand\t{\theta}\def\vt{\vartheta}
\newcommand{\vphi}{\varphi}
\def\z{\zeta}

\newcommand{\tsigma}{\tilde{\s}}
\newcommand{\tbsigma}{\tilde{\bsigma}}
\def\te{\tilde{\e}}
\def\tu{\tilde{u}}

\newcommand{\bchi}{\mbox{\boldmath$\chi$}}
\newcommand{\bdelta}{\mbox{\boldmath$\delta$}}
\newcommand{\bepsilon}{\mbox{\boldmath$\epsilon$}}
\newcommand{\bfeta}{\mbox{\boldmath$\eta$}}
\newcommand{\bgamma}{\mbox{\boldmath$\gamma$}}
\newcommand{\bomega}{\mbox{\boldmath$\omega$}}
\newcommand{\bvphi}{\mbox{\boldmath$\varphi$}}
\newcommand{\bphi}{\mbox{\boldmath$\phi$}}
\newcommand{\bPhi}{\mbox{\boldmath$\Phi$}}
\newcommand{\bpsi}{\mbox{\boldmath$\psi$}}
\newcommand{\bPsi}{\mbox{\boldmath$\Psi$}}
\newcommand{\bsigma}{\mbox{\boldmath$\sigma$}}
\newcommand{\btau}{\mbox{\boldmath$\tau$}}
\newcommand{\bxi}{\mbox{\boldmath$\xi$}}
\newcommand{\brho}{\mbox{\boldmath$\rho$}}

\newcommand{\bbeta}{\mbox{\boldmath$\beta$}}
\newcommand{\bzeta}{\mbox{\boldmath$\zeta$}}

\def\bk{\boldsymbol{\kappa}}
\def\bmu{\boldsymbol\mu}
\def\bxi{\boldsymbol{\xi}}
\def\bz{\boldsymbol{\zeta}}
%%%%%%%%

%English Letters
\def\ba{{\bf a}}
\def\bb{{\bf b}}
\def\bc{{\bf c}}
\def\be{{\bf e}}
\def\bff{{\bf f}}
\def\bg{{\bf g}}
\def\bn{{\bf n}}
\def\bp{{\bf p}}
\def\bq{{\bf q}}
\def\bs{{\bf s}}
\def\bt{{\bf t}}
\def\bu{{\bf u}}
\def\bv{{\bf v}}
\def\bw{{\bf w}}
\def\bx{{\bf x}}
\def\by{{\bf y}}
\def\bzz{{\bf z}}

\def\bD{{\bf D}}
\def\bE{{\bf E}}
\def\bF{{\bf F}}
\def\bH{{\bf H}}
\def\bJ{{\bf J}}
\def\bV{{\bf V}}
\def\bU{{\bf U}}
\def\bW{{\bf W}}
\def\bX{{\bf X}}
\def\bY{{\bf Y}}

\def\cA{{\cal A}}
\def\cC{{\cal C}}
\def\cD{{\cal D}}
\def\cE{{\cal E}}
\def\cF{{\cal F}}
\def\cG{{\cal G}}
\def\cI{{\cal I}}
\def\cJ{{\cal J}}
\def\cK{{\cal K}}
\def\cL{{\cal L}}
\def\cO{{\cal O}}
\def\cP{{\cal P}}
\def\cQ{{\cal Q}}
\def\cR{{\cal R}}
\def\cS{{\cal \Sigma}}
\def\cT{{\cal T}}
\def\cU{{\cal U}}
\def\cV{{\cal V}}

\def\scT{{_\cT}}
\def\sD{{_D}}
\def\sE{{_E}}
\def\sF{{_F}}
\def\sFz{{_{F_z}}}
\def\sK{{_K}}
\def\sI{{_I}}
\def\sb{{_b}}
\def\sN{{_N}}

\def\curl{{{\bf curl} \ }}
\def\rot{{\mbox{rot}\ }}
\def\BPI{{\bf \Pi}}

\def\cth{\cT_h}
\def\ctH{\cT_H}

\def\tJ{\tilde{\J}}

\def\hK{\widehat{K}}
\def\hx{\widehat{x}}
\def\hy{\widehat{y}}
\def\bhv{\widehat{\bv}}
%%%%%%%%%

%math symbols
\def\l{\ell}
\def\bl{\boldsymbol{\ell}}
\def\col{\colon}
\def\f12{\frac12}
\def\dfrac{\displaystyle\frac}
\def\dint{\displaystyle\int}
\def\nab{\nabla}
\def\p{\partial}
\def\sm{\setminus}
\def\dsum{\displaystyle\sum}
\newcommand{\pp}[2]{\frac{\partial {#1}}{\partial {#2}}}
\def\bzero{{\bf 0}}

\def\divv{\nab\cdot}
\def\divx{\nab_x\cdot}
\def\divtx{\nab_{t,x}\cdot}
\def\nabx{\nab_x}

\newcommand{\grad}{\nabla}
\newcommand{\curlt}{{\nabla \times}}
\newcommand{\gperp}{\nabla^{\perp}}
\newcommand{\gradt}{\nabla\cdot}

\def\forallqq{\quad\forall\,}
\def\aph{A^{1/2}}
\def\amh{A^{-1/2}}

\def\osc{{\rm osc \, }}

\def\Im{{\rm Im}}
\newcommand{\tr}{{\rm tr}}
\def\divvr{{\rm div}}
\def\curllr{{\rm curl}}
\def\curll{{\rm curl}}
\def\curl{{\bf curl}}
\newcommand{\bgrad}{{\bf grad}}
\newcommand\diam{\mathrm{diam\,}}
\renewcommand\Im{\mathrm{Im\,}}
\def\Span{\mbox{Span}}
\def\supp{\mbox{supp\,}}
\newcommand{\trace}{{\rm trace}}

\newcommand{\tri}{|\!|\!|}
\newcommand{\ljump}{\lbrack\!\lbrack}
\newcommand{\rjump}{\rbrack\!\rbrack}
%%%%%%%%%%%%%%%%%%%%%%%%%%%%%%%%%%%%%%%%%%%%%%
\newcommand{\bdm}{\begin{displaymath}}
\newcommand{\edm}{\end{displaymath}}
\newcommand{\beq}{\begin{equation}}
\newcommand{\eeq}{\end{equation}}
\newcommand{\beqa}{\begin{eqnarray}}
\newcommand{\eeqa}{\end{eqnarray}}
\newcommand{\beqas}{\begin{eqnarray*}}
\newcommand{\eeqas}{\end{eqnarray*}}
%misc
\newcommand{\ul}{\underline}
\newcommand{\wh}{\widehat}
\newcommand{\la}{\langle}
\newcommand{\ra}{\rangle}

%Spaces
\newcommand{\Lt}{L^2(\Omega)}
\newcommand{\Lts}{L^2(\Omega)^2}
\newcommand{\Ltc}{L^2(\Omega)^3}
\newcommand{\Ho}{H^1(\Omega)}
\newcommand{\Hoh}{H^1(\wh{\Omega})}
\newcommand{\Hoi}{H^1(\Omega_i)}
\newcommand{\Hos}{H^1(\Omega)^2}
\newcommand{\Hoc}{H^1(\Omega)^3}
\newcommand{\Hoch}{H^1(\wh{\Omega})^3}
\newcommand{\Hoci}{H^1(\Omega_i)^3}
\newcommand{\Hoz}{H^1_0(\Omega)}
\newcommand{\Ht}{H^2(\Omega)}
\newcommand{\Hti}{H^2(\Omega_i)}
\newcommand{\Hts}{H^2(\Omega)^2}
\newcommand{\Htc}{H^2(\Omega)^3}
\newcommand{\Htz}{H^0(\Omega)}
\newcommand{\Hh}{H^{1/2}(\Gamma)}
\newcommand{\Hhi}{H^{1/2}(\Gamma_i)}
\newcommand{\Hmh}{H^{-1/2}(\Gamma)}
\newcommand{\Hdiv}{H(\divvr;\,\Omega)}
\newcommand{\Hdivh}{H(\divv;\,\wh \Omega)}
\newcommand{\hcurl}{H(\curl\,A;\,\Omega)}
\newcommand{\Hcurl}{H(\curll\,A;\,\Omega)}
\newcommand{\Hcrl}{H(\curll\,;\,\Omega)}
\newcommand{\hcrl}{H(\curl\,;\,\Omega)}
\newcommand{\Hcrlh}{H(\curll\,;\,\wh\Omega)}
\newcommand{\hcrlh}{H(\curl\,;\,\wh\Omega)}
\newcommand{\Wdiv}{\BW_0(\mbox{\divv}\,;\,\Omega)}
\newcommand{\Wcurl}{\BW_0(\mbox{\curl}\,A;\,\Omega)}
\newcommand{\WcrossV}{\BW \times V}